\documentclass[12pt,a4paper]{amsart}
\usepackage{amscd}
\usepackage{amssymb}
\usepackage[centering,text={15.5cm,22cm}]{geometry}
\usepackage{graphicx,color}
\usepackage[all]{xy}
\usepackage{mathrsfs}
\usepackage{marvosym}
\usepackage{stmaryrd}
\usepackage{srcltx}
\usepackage{hyperref}

\definecolor{shadecolor}{rgb}{1,0.9,0.7}

\newtheorem{theorem}{Theorem}[section]
\newtheorem{lemma}[theorem]{Lemma}
\newtheorem{lemma-definition}[theorem]{Lemma-Definition}
\newtheorem{proposition}[theorem]{Proposition}
\newtheorem{corollary}[theorem]{Corollary}

\theoremstyle{definition}
\newtheorem{setup}[theorem]{Setup}
\newtheorem{definition}[theorem]{Definition}

\newtheorem{convention}[theorem]{Convention}
\newtheorem{example}[theorem]{Example}

\newtheorem*{acknowledgement}{Acknowledgement}

\theoremstyle{remark}
\newtheorem{remark}[theorem]{Remark}

\numberwithin{equation}{section}
\numberwithin{figure}{section}

\newcommand {\lfor} {\llbracket}
\newcommand {\rfor} {\rrbracket}

\newcommand{\DD} {\mathbb{D}}
\newcommand{\HH} {\mathbb{H}}
\newcommand{\NN} {\mathbb{N}}
\newcommand{\ZZ} {\mathbb{Z}}
\newcommand{\QQ} {\mathbb{Q}}
\newcommand{\RR} {\mathbb{R}}
\newcommand{\CC} {\mathbb{C}}

\newcommand{\PP} {\mathbb{P}}
\renewcommand{\AA} {\mathbb{A}}

\newcommand {\shD}  {\mathcal{D}}

\newcommand {\shF}  {\mathcal{F}}

\newcommand {\shM}  {\mathcal{M}}

\newcommand {\shO}  {\mathcal{O}}

\newcommand {\shU}  {\mathcal{U}}

\newcommand {\shW}  {\mathcal{W}}
\newcommand {\shX}  {\mathcal{X}}

\newcommand {\shZ}  {\mathcal{Z}}

\newcommand {\foX}  {\mathfrak{X}}

\newcommand {\fob}  {\mathfrak{b}}

\newcommand {\fop}  {\mathfrak{p}}

\newcommand {\fou}  {\mathfrak{u}}

\newcommand {\an}  {\mathrm{an}}

\newcommand {\bary} {{\operatorname{bar}}}

\newcommand {\can} {{\mathrm{can}}}

\newcommand {\codim} {\operatorname{codim}}

\newcommand {\dlog} {\operatorname{dlog}}

\newcommand {\eps}  {\varepsilon}

\newcommand {\GL}  {\operatorname{GL}}

\newcommand {\gp}  {{\operatorname{gp}}}
\newcommand {\Gr}  {\operatorname{Gr}}

\newcommand {\Hom}  {\operatorname{Hom}}

\newcommand {\hra} {\hookrightarrow}

\newcommand {\Int}  {\operatorname{Int}}

\newcommand {\limdir} {\varinjlim}

\newcommand {\lra}  {\longrightarrow}
\newcommand {\ls}  {\dagger}

\renewcommand {\max} {{\operatorname{max}}}

\renewcommand{\O}  {\mathcal{O}}

\newcommand {\ori} {\operatorname{or}}

\renewcommand{\P}  {\mathscr{P}}

\newcommand {\ra}  {\to}

\newcommand {\Spec} {\operatorname{Spec}}
\newcommand {\Spf}  {\operatorname{Spf}}

\newcommand {\triang} {\triangle}

\newcommand {\trop}  {{\operatorname{trop}}}

\newcommand {\ul} {\underline}

\def\mydate{\ifcase\month \or January\or February\or March\or
April\or May\or June\or July\or August\or September\or October\or 
November\or December\fi \space\number\day,\space\number\year}

\newcommand {\Tbs} {T}

\begin{document}

%===========================================================

\title
[Canonical coordinates in toric degenerations]
{Canonical coordinates in toric degenerations}
%\mbox{\tiny -preliminary version-}}
\author{Helge Ruddat, Bernd Siebert}

% author one information
%\author{Helge Ruddat} 
\address{\tiny JGU Mainz, Institut f\"ur Mathematik, Staudingerweg 9, 55099 Mainz, Germany}
%\curraddr{}
%\thanks{This work was partially supported by DFG research grant RU 1629/1-1}
\email{ruddat@uni-mainz.de}

% author two information
%\author{Bernd Siebert} 
\address{\tiny Universit\"at Hamburg, Fakult\"at f\"ur Mathematik,
Bundesstra\ss e 55, 20146~Hamburg, Germany}
%\curraddr{}
\email{bernd.siebert@math.uni-hamburg.de}
\thanks{This work was partially supported by HR's Carl-Zeiss Postdoctoral Fellowship}

\maketitle
\setcounter{tocdepth}{1}
\tableofcontents
\bigskip

%===========================================================
%===========================================================
\section*{Introduction}
Mirror symmetry suggests to study families of varieties with a
certain maximal degeneration behaviour \cite{CdGP91}, \cite{Mo93},
\cite{De93}, \cite{HKTY95}. In the important Calabi-Yau case this
means that the monodromy transformation along a general loop around
the critical locus is unipotent of maximally possible exponent
\cite[\S2]{Mo93}. The limiting mixed Hodge structure on the
cohomology of a nearby smooth fibre is then of Hodge-Tate type
\cite{De93}.

An important insight in this situation is the existence of a
distinguished class of holomorphic coordinates on the base space of
the maximal degeneration \cite{Mo93}, \cite{De93}. Explicitly, these
\emph{canonical coordinates} are computed as $\exp$ of those period
integrals of the holomorphic $n$-form $\Omega$ over $n$-cycles that
have a logarithmic pole at the degenerate fibre. For an algebraic
family they are often determined as certain solutions of the
Picard-Fuchs equation solving the parallel transport with respect to
the Gau\ss-Manin connection. For complete intersections in toric
varieties these solutions can be written as hyper-geometric series.
In particular, canonical coordinates are typically transcendental
functions of the algebraic parameters. The coordinate change from
the algebraic parameters to the canonical coordinates is referred to
as \emph{mirror map}. The explicit determination of the mirror
map is an indispensable step in equating certain other period
integrals with generating series of Gromov-Witten invariants on the
mirror side. 

The purpose of the present paper is to address the topic of
canonical coordinates in the toric degeneration approach to mirror
symmetry developed by Mark Gross and the second author
\cite{logmirror1}, \cite{logmirror2}, \cite{affinecomplex}. In this
program, \cite[Corollary~1.31]{affinecomplex} provides a canonical class
of degenerations  defined over completions of affine toric
varieties. Our main result says that the toric monomials of the base
space are canonical in the above sense. In other words, the mirror
map is trivial. This is another important hint of the
appropriateness of the toric degeneration approach. In particular,
we expect that the Gromov-Witten invariants of the mirror are rather
directly encoded in the wall structure of \cite{affinecomplex}.
Another consequence of our result is that the formal smoothings
constructed in \cite{affinecomplex} lift to analytic families. In
order to prove this, we construct sufficiently many cycles using
tropical methods. The computation of the period integrals over
these we then carry out explicitly.

Morrison \cite{Mo93} defines canonical coordinates as follows. Let
$f:\shX\ra T$ be a maximal degenerating analytic family of
Calabi-Yau varieties. The fibre over $t\in T$ is denoted $X_t$ and
the central fibre of the degeneration lies over $0\in T$. Let
$D\subset T$ be the critical locus of $f$ where the fibres $X_t$ are
singular. Assuming $T$ smooth and $D$ to have simple normal
crossings denote by $T_1,\ldots,T_r$ the monodromies around the
irreducible components of $D$. The endomorphism given as any
positive linear combination of $\log T_1,\ldots,\log T_r$ defines
the weight filtration  $0\subset W_0\subset W_2\subset \ldots$ on
$H_n(X_t,\ZZ)$ for any fixed $t\not\in D$.  A vanishing $n$-cycle
$\alpha\in W_0$ is a generator of $W_0$. It is unique up to sign.
Let $\Omega$ be a non-vanishing section of
$\Omega^n_{\shX/T}(\log(\shD))$, a relative holomorphic volume form
with logarithmic poles along $\shD=\pi^{-1}(D)$. The fibrewise
integral of $\Omega$ over the parallel transport of an element in
$W_{2k}$ yields a function on $T$ with a logarithmic pole of order
at most $k$. Hence the following definition makes sense.

\begin{definition}[Canonical coordinates]
\label{def-cancoord}
Given $\beta\in W_2$ one defines a meromorphic function $h_\beta$ on
the base $\Tbs$ by
\[
h_{\beta}(t) = \exp\left(-2\pi i \frac{\int_\beta
\Omega}{\int_\alpha \Omega}\right), \quad t\in T\setminus D.
\]
Note that taking $\exp$ disposes of the ambiguity of the
monodromy around $X_0$ which adds multiples of $\alpha$ to $\beta$.
If $h_\beta$ extends as a holomorphic function to $T$ it is called a
\emph{canonical coordinate}.
\end{definition}

We consider the canonical degenerations given in
\cite{affinecomplex}. The central fibre $X_0$ is constructed from a
polarized tropical manifold $(B,\P,\varphi)$ and then a formal
degeneration $\foX\ra \Spf\CC\lfor t\rfor$ with central fibre $X_0$ is
obtained by a deterministic algorithm that takes as input a log
structure on $X_0$. Mumford's degenerations of abelian varieties
\cite{Mum72} are examples of such canonical degenerations.
Degenerating a Batyrev-Borisov Calabi-Yau manifold \cite{BB94} into
the toric boundary \cite{Gr05} gives another important example of
degenerations with the type of special fibre considered here, with a
priori non-canonical algebraic deformation parameters. One obtains
(formal) canonical families here by reconstructing the family up to
base change from the central fibre via \cite{affinecomplex} (with
higher-dimensional parameter space). The resulting base coordinate
then coincides with Morrison's canonical coordinates in
Definition~\ref{def-cancoord} as follows from the results of this
paper. The transformation from the algebraic to the transcendental
coordinate is the aforementioned mirror map.

The definition of $(B,\P,\varphi)$, which we recall in
\S\ref{subsec-intcomplex}, can be found in \cite[\S4.2]{logmirror1}.
Here $B$ is a real $n$-dimensional affine manifold with singular
locus $\Delta$ at most in codimension two. The linear part of its
holonomy is integral. The affine manifold comes with a decomposition
$\P$ into integral polyhedra and a multi-valued piecewise affine
function $\varphi:B\ra\RR$. The toric varieties given by the
lattice polytopes of $\P$ are the toric strata of $X_0$. The
singular locus $\Delta$ is part of the codimension two skeleton of the
barycentric subdivision of $\P$. The function $\varphi$ encodes the
discrete part of the log structure, namely toric local
neighbourhoods of $X_0$ in $\shX$, each given by a cone that is the
upper convex hull over $\varphi$ on a local patch of $B$.

For $k\in\NN$, let $X_k$ be the canonical smoothing of $X_0$ to
order $k$ constructed in \cite{affinecomplex}. If $X_0$ is
projective then the formal degeneration is induced by a formal
family $\hat\shX\to \Spec(\CC\lfor t\rfor)$ of schemes.\footnote{Details
for this statement without the cohomological assumptions of
\cite{affinecomplex}, Corollary~1.31, will appear in \cite{theta}} In
any case, at least if $X_0$ is compact, there exists an analytic
family $\shX\ra\Tbs$ whose restriction to order $k$ coincides with
$X_k$ (Theorem~\ref{Thm: versal}). We assume
$B$ to be oriented. We define tropical $1$-cycles in $B$ and show
how each such determines an $n$-cycle in the nearby fibres $X_t$ of
$X_0$ in $\shX$, unique up to adding a vanishing $n$-cycle. Under
the simplicity assumption \S\ref{subsec-simplicity}, we prove that
the tropically constructed $n$-cycles generate $W_2/W_0$, the
relevant graded piece of the monodromy weight filtration. We then
integrate the canonical $n$-form $\Omega$ on $\foX$ over these
cycles and compute the exponential of the result to order $k$ around
$0\in T$. Thus despite the logarithmic pole of the integral it makes
sense to talk about canonical coordinates for the formal family
$\foX\to \Spf(\CC\lfor t\rfor)$. The precise statement of the Main Theorem
below (Theorem~\ref{period-computed}) requires some explanations
that we now turn to.

Let $\Lambda$ and $\check\Lambda$ denote the local systems (stalks
isomorphic to $\ZZ^n$) of flat integral tangent vectors on
$B\setminus \Delta$. Let $i_*\Lambda$ and $i_*\check\Lambda$ be
their pushforward to $B$ (these are constructible sheaves). As
described in \cite[\S2.1]{logmirror1}, $X_0$ itself can be
reconstructed from $(B,\P,\varphi)$ together with an element $s$ in
$H^1(B,i_*\check\Lambda\otimes\CC^\times)$, see
\S\ref{subsec-gluing}. This one-cocycle is represented by a
collection $(s_{\tau_0\subset\tau_1})$ for $\tau_0,\tau_1\in\P$ and
is called \emph{gluing data}. If furthermore $B$ is \emph{simple}
(see \S\ref{subsec-simplicity}) then the log structure on
$X_0^\dagger$ is determined by the gluing data, see
\cite[Proposition~4.25, Theorem~5.2]{logmirror1}. Hence, in the
simple case, one may view $H^1(B,i_*\check\Lambda\otimes\CC^\times)$
as the moduli space of log structures on $X_0$.
\begin{definition}
\label{def-trop-1-cycle}
A tropical 1-cycle $\beta_\trop$ in $B$ is a graph with oriented
edges embedded in $B\setminus\Delta$ whose edges $e$ are labelled by
a non-trivial section $\xi_e\in \Gamma(e,\Lambda|_e)$. It is subject
to the following conditions. Its vertices lie outside the
codimension one skeleton $\P^{[n-1]}$ of $\P$ and its edges
intersect $\P^{[n-1]}$ in the interior of codimension one cells in
isolated points. A vertex is univalent if and only if it is
contained in $\partial B$. Finally, at each vertex $v$, the
following \emph{balancing condition} holds
\begin{equation}
\label{balancing}
\sum_{v\in e} \eps_{e,v} \xi_e=0.
\end{equation}
Here $\eps_{e,v}\in\{-1,1\}$ is the orientation of $e$ at $v$. 
\end{definition}

Similar cycles have been known in the theory of completely
integrable Hamiltonian systems, see \cite[Theorem~7.4]{Sy}. In the
context of the Gross-Siebert program, similar tropical cycles have
been used by \cite{CBM13}. The balancing condition is a typical
feature in tropical geometry, see \cite{Mi05}. 

\begin{example} (Tropical cycles from the $1$-skeleton)
Let $\P^{[1]}$ denote the set of one-dimensional cells in $\P$.
For a vertex $v\in\P$ with $\omega\in\P^{[1]}$ an edge containing it, we denote by $d_{v,\omega}$ the primitive integral tangent vector to $\omega$ pointing from $v$ into $\omega$.
Then for any weight function
$a:\P^{[1]}\to\ZZ$ and any vertex $v\in\P$ we can check the analogue
of the balancing condition~\eqref{balancing} at $v$:
\[
\sum_{\omega\ni v} a(\omega) d_{v,\omega}=0.
\]
Assuming this balancing condition holds for every $v$ we can then
define a tropical $1$-cycle by taking the graph with edges
$\{\omega\in\P^{[1]}\,|\, a(\omega)\neq 0\}$ and the embedding into
$B \setminus\Delta$ a small perturbation of the $1$-skeleton to make
the resulting cycle disjoint from $\Delta$ and its intersection with
$\P^{[n-1]}$ discrete.  To define the section $\xi_e$ and the
orientation for the edge $e$, we choose a vertex $v$ of every edge
$\omega$. Now the section $\xi_e\in\Gamma(e,\Lambda)$ of the edge
$e$ of $\beta_\trop$ arising as a perturbation of
$\omega\in\P^{[1]}$ is defined by parallel transport of
$a(\omega)\cdot d_{v,\omega}$ and we orient $e$ by $d_{v,\omega}$. 
Note that $d_{v,\omega}$ is invariant under local monodromy around
$\Delta$, so local parallel transport is uniquely defined. Choosing
the other vertex of $\omega$ instead results in a double sign
change, namely in the orientation of $e$ as well as in the section
$\xi_e$ and so the choice of vertex $v$ is insignificant.

A special case of this example arises if $(B,\P,\varphi)$ is the
dual intersection complex of a degeneration with normal crossing
special fibre.  The one-skeleton at each vertex then looks like the
fan of projective space. As the primitive generators of the rays in
this fan are balanced, any non-trivial constant weight function
$w:\P^{[1]}\ra\ZZ$ yields a tropical $1$-cycle by the above
procedure. This way, one obtains a generator for $W_2/W_0$ for the
mirror dual Calabi-Yau of a degree $(n+1)$-hypersurface in $\PP^n$,
e.g. the mirror dual of the quintic threefold.
\end{example}

We associate to a tropical $1$-cycle $\beta_\trop$ an $n$-cycle
$\beta\in H_n(X_t,\ZZ)$ in the nearby fibres $X_t, t\neq 0$, see
\S\ref{section-tropical-to-homology}. The association $\beta_\trop
\leadsto \beta$ is canonical up to adding a multiple of the
vanishing $n$-cycle $\alpha$. An oriented basis $v_1,\ldots,v_n$ of
a stalk of $\Lambda$ gives a global $n$-form 
\[
\Omega= \dlog z^{v_1}\wedge \ldots\wedge \dlog z^{v_n}
\]
on $X_0^\dagger$ which extends canonically to $\shX$ as a section of
$\Omega^n_{\shX^\dagger/\Tbs^\dagger}$ by requiring that its
integral over the vanishing $n$-cycle is constant. Now the vanishing
$n$-cycle on $X_t$ is homologous in $\shX$ to the $n$-torus
$|z^{v_i}|={\rm const}$, $i=1,\ldots,n$, in $X_0$. Hence the
constant is computed to be
\[
\int_\alpha \Omega=(2\pi i)^n.
\]

The multi-valued piecewise affine function $\varphi$ is uniquely
determined by a set of positive integers $\kappa_\rho$ telling the
change of slope for each codimension one cell $\rho\in\P^{[n-1]}$.
This is defined as follows. Let $\sigma_+$, $\sigma_-\in\P^{[n]}$
be the two maximal cells containing $\rho$. Let
$d_\rho\in\check\Lambda_{\sigma_+}$ be the primitive normal to
$\rho$ that is non-negative on $\sigma_+$. In particular, the
tangent space to $\rho$ is $d_\rho^\perp$. We have
$\varphi|_{\sigma_\pm}$ is affine, say the linear part is given by
$m_+$ and $m_-$, respectively. Their difference needs to be a
multiple of $d_\rho$. Thus there exists $\kappa_\rho\in
\NN\setminus\{0\}$, called the \emph{kink of $\varphi$ at }$\rho$,
with
\begin{equation}
\label{kink}
m_+|_\rho-m_-|_\rho = \kappa_\rho d_\rho.
\end{equation}
Somewhat more generally with a view towards \cite{theta}, let
$\tilde\P^{[n-1]}$ denote the set of those codimension one cells of
the barycentric subdivision of $\P$ that lie in codimension one
cells of $\P$. In this case we admit a different $\kappa_{
\ul{\rho}}$ for each $\ul{\rho} \in\tilde\P^{[n-1]}$. Taking
$\Delta$ to be the union of the boundaries of elements of
$\tilde\P^{[n-1]}$ we may assume that $\beta$ meets any $\ul\rho\in
\tilde\P^{[n-1]}$ at most in its relative interior. The logic of
this notation is $\rho$ is the codimension one cell of $\P$
containing a $\ul \rho\in\tilde\P^{[n-1]}$. For the purpose of
\cite{affinecomplex}, the transition from $\P^{[n-1]}$ to
$\tilde\P^{[n-1]}$ is unnecessary as in this setup all $\kappa_{
\ul{\rho}}$ agree for a given $\rho$.

The main results are the following.
\begin{theorem} 
\label{period-computed}
Let $B$ be oriented and assume $\beta_\trop\cap\partial B=\emptyset$
and $\beta_\trop$ is compact. Then we have
\[
 h_\beta(t) = (-1)^\nu\prod_{p\in\beta_\trop\cap
\tilde\P^{[n-1]}}  s_p t^{\kappa_p\langle \xi_{e_p},d_p\rangle}.
\]
where 
\begin{itemize}
\item[$\nu$] denotes the sum of the valencies of all the vertices of $\beta_\trop$.
\item
[$\langle\cdot,\cdot\rangle$] is the pairing of tangent vectors
$\Lambda$ and co-tangent vectors $\check\Lambda$, 
\item
[$\kappa_p$] $\in\ZZ_{>0}$ is the kink of $\varphi$ at the
codimension one cell $\ul\rho\in\tilde\P^{[n-1]}$ containing $p$, 
\item
[$e_p$] is the edge of $\beta$ containing $p$,
\item
[$d_p$] is the primitive normal to $\rho$ that pairs positively with
an oriented tangent vector to $e_p$ at $p$ and
\item
[$s_p$] $\in\CC^\times$ is determined by $\xi_{e_p}$, the gluing
data $s=(s_{\tau_0\subset\tau_1})$ and the orientation of $e_p$ at
$p$, see \eqref{def-sp} and Definition~\ref{opengluingdata}.
\end{itemize}
Hence, up to an explicit constant factor and taking a power, $t$ is the
canonical coordinate of \cite{Mo93}.
\end{theorem}

\begin{proof}
The proof occupies \S\ref{section-compute-periods}.
\end{proof}

\begin{remark}[Higher dimensional base $\Tbs$]
It is straightforward to generalize Theorem~\ref{period-computed} to
the case where $\dim\Tbs>1$. The base monoid $\NN$ gets replaced by
a monoid $Q$ and $t^{\kappa_p}=z^{\kappa_p}$ gets replaced by
$z^{q_p}$ with $q_p\in Q$, see \cite[Appendix]{theta}. The adaption
of our proofs to this case is straightforward. Alternatively, one
can deduce the multi-parameter case from the one-parameter case
because a function is monomial if and only if its base change to any
monomially defined one-parameter family is monomial.
\end{remark}

\begin{remark}[Boundary and compactness]
If $\beta_\trop\cap\partial B\neq \emptyset$ then $\Omega$ acquires
a logarithmic pole on $\beta$, so the integral $\int_\beta\Omega$ is
not finite. The integral is also infinite if $\beta_\trop$ is
non-compact (necessarily $B$ is non-compact then as well).
\end{remark}

It remains to understand in which cases the cycles $\beta\in W_2$
obtained from tropical $1$-cycles $\beta_\trop$ actually generate
$W_2/W_0$. Let $C_1(B,i_*\Lambda)$ denote the group of tropical
$1$-cycles.
\begin{definition} 
\label{def-enough-t1c}
Let $(B,\P,\varphi)$ be a polarized tropical manifold. We say that
$(B,\P,\varphi)$ has \emph{enough tropical $1$-cycles} if the set 
$\{\beta\mid \beta_\trop\in C_1(B,i_*\Lambda)\}$ generates
$W_2/W_0$. \end{definition}

\begin{theorem} 
\label{enough-cycles}
\begin{enumerate}
\item
Let $C_1(B,i_*\Lambda)$ denote the group of tropical $1$-cycles.
The natural map
\[
C_1(B,i_*\Lambda)\lra H_1(B,\partial B;i_*\Lambda)
\]
associating to a tropical 1-cycles its homology class in sheaf
homology is surjective.
\item

If $B$ is oriented, we have a canonical isomorphism
\[
 H^{n-1}(B,i_*\Lambda)=H_1(B,\partial B;i_*\Lambda) 
\]
\end{enumerate}
\end{theorem}

\begin{proof}
(1) is Theorem~\ref{shape-of-cycles} and (2) is
Theorem~\ref{main-affine-result},(1).
\end{proof}

Via Hodge theory of toric degenerations \cite{logmirror2,Ru10}, we
will deduce as Corollary~\ref{genW2W0} the following result from
Theorem~\ref{enough-cycles} in \S\ref{sec-W2W0}.

\begin{theorem}
If $(B,\P,\varphi)$ is simple then it has enough tropical $1$-cycles.
\end{theorem}

\begin{remark}[Beyond simplicity]
By \cite[Example~1.16]{Ru10} it is known that
\eqref{injection-affine-to-log} might not be an isomorphism beyond
simplicity. For example for the quartic degeneration in $\PP^3$ to
the union of coordinate planes, the left-hand side of
\eqref{injection-affine-to-log} has rank $2$ whereas the right-hand
side has rank $20$. To turn \eqref{injection-affine-to-log} into an
isomorphism, one needs to degenerate further, see
\S\ref{subsec-simplicity}. Simplicity is closely related to making
the tropical variety of the quartic family smooth in the sense of
tropical geometry \cite{Mi05}.
\end{remark}

\begin{acknowledgement} 
We would like to thank Mark Gross, Duco van Straten and Eric Zaslow
for useful discussions.
\end{acknowledgement}

\begin{convention} 
We work in the complex analytic category. Every occurrence of $\Spec
A$ for a $\CC$-algebra $A$ is implicitly to be understood as the
analytification $(\Spec A)_\an$ of the $\CC$-scheme $\Spec A$.
\end{convention}

%===========================================================
%
%		Key example: the elliptic curve
%
%===========================================================

\section{Key example: the elliptic curve}
\label{Sect: elliptic curve}
As an illustration, we compute the periods for the nodal
degeneration of an elliptic curve. The technique we use for the
computation of periods of higher-dimensional Calabi-Yau manifolds is
a generalization of how we do it here. We denote the multiplicative
group of complex numbers by $\CC^*$ and consider the Tate family of
elliptic curves which is the (multiplicative) group quotient
\[
E_t = \CC^*/t^{k\ZZ}
\]
for $0<|t|<1$ and $k\in\ZZ_{>0}$. If $z$ denotes the standard
coordinate on $\CC$, we define $\Omega=\dlog z=\frac{dz}z$. This
$1$-form is invariant under $z\mapsto \lambda z$ for
$\lambda\in\CC^*$ and hence it descends to the Tate family. We have
two natural cycles coming from the description. Let
$\alpha(s)=e^{is}$ be a counterclockwise loop around the missing
origin in $\CC^*$, we find
\begin{equation}
\label{Eq: alpha-int Tate curve}
\int_\alpha\Omega = \int_{0}^{2\theta} e^{-i\theta}i
e^{i\theta}d\theta=2\pi i
\end{equation}
is independent of $t$. In the completed family below $\alpha$ is
going to be a vanishing cycle and \eqref{Eq: alpha-int Tate curve}
shows $\Omega$ is the canonical holomorphic volume form.
\begin{figure}
\resizebox{0.4\textwidth}{!}{
\includegraphics{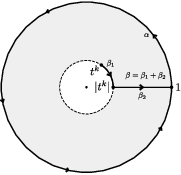}
}
\caption{Fundamental region of $\CC^*/t^{k\ZZ}$}
\label{Fig: fund-region}
\end{figure}
The other cycle $\beta$ is depicted in Figure~\ref{Fig:
fund-region}. We write $t=r e^{i\psi}$. Splitting $\beta$ in an
angular part $\beta_1$ and a radial part $\beta_2$, we
compute\footnote{Note that $\beta$ can not be defined consistently
in the whole family; the various choices differ homologically by
multiples of $\alpha$ and lead to different branches of $\log t^k$.}
\begin{equation}
\label{intbeta}
\int_\beta\Omega = \int_{\beta_1}\Omega + \int_{\beta_2}\Omega =
\int^{0}_{k\psi} e^{-i\theta}i e^{i\theta}d\theta + \int_{r^k}^1
r^{-1}dr= -ik\psi -k\ln r = -\log t^k.
\end{equation}
The canonical coordinate is given as 
\[
\exp\left(-2\pi i\frac{\int_\beta\Omega}{\int_\alpha\Omega}\right) = t^k.
\]
For the purpose of generalizing this computation to higher
dimensional Calabi-Yau manifolds that are not necessarily complex
tori, we next recompute $\int_\beta\Omega$ somewhat differently. In
terms of the Tate family, this means that we focus our attention on
the (yet missing) central fibre. The family can be completed over
the origin by an $I_1^k$ type nodal rational curve as follows. (An
$I_k$ fibre is also possible, cf.~\cite[\S VII]{DR73}.) Consider the
action of the group $\mu_k$ of $k$th roots of unity on $\AA^2$ given
by $\zeta.(x,y)=(\zeta x,\zeta^{-1}y)$. Let $U$ be the open subspace
of the quotient
\[
\AA^2/\mu_k=\Spec \CC[x^k,y^k,xy] = \Spec \CC[z,w,t]/(zw-t^k).
\] 
defined by $|t|<1$. Set $V=(\AA^1\setminus\{0\})
\times\{t\in\CC\,|\,|t|<1\}\subset \AA^1\setminus\{\}\times\AA^1$.
Define $\shX$ to be the quotient in the analytic category defined by
the \'etale equivalence relation (pushout)
\[
\vcenter{\vbox{
%\resizebox{0.9\textwidth}{!}{
\xymatrix@C=30pt
{ 
V\ar@<.6ex>^{(u,t)\mapsto(u,u^{-1}t^k,t)}[rr]
\ar@<-.6ex>_{(u,t)\mapsto (ut^k,u^{-1},t)}[rr] && U \ar@{.>}[r]&
\shX}
%}
}}
\]
The map $f:U\ra\AA^1,\ (z,w,t)\mapsto t$ descends to $\shX$. Define
$X_t=f^{-1}(t)$.  For $t\neq 0$ fixed, we find $X_t$ is the
hypersurface of $U$ given by $z=t^kw^{-1}$ modulo the equivalence
relation $w^{-1}=z$. So indeed $X_t=E_t$ and $\shX_\DD=f^{-1}(\DD)$
is a completion of the Tate family over the origin. By abuse of
notation, we will set $\shX=\shX_\DD$ now.

We next turn to the form to integrate. For this we choose a
generator $\Omega$ of the trivial bundle $\Omega^1_{\shX/\DD}(\log
X_0)\cong \shO_\shX$.  There is a canonical generator (up to sign)
as before. Namely on $X_0$, take $\Omega|_{X_0}=
\frac{dz}{z}=-\frac{dw}{w}$ and lift this by setting
$\Omega=\frac{dz}{z}$, now on $\shX$.  The restriction of $\Omega$
to $X_t$ for $t\neq 0$ coincides with the $\Omega$ considered above
when we computed the periods.  The cycle $\alpha$ is now identified
with the vanishing $1$-cycle of the degeneration of $X_t$ as $t\ra
0$.  There is only one such integral also when we go to dimension
$n$ where $\alpha$ is then homeomorphic to $(S^1)^n$. More
interesting is the period integral over $\beta$ of which there might
be several in higher dimensions. What we are going to do is
construct first a tropical version $\beta_\trop$ of $\beta$ in the
intersection complex of $X_0$. For the present degeneration of
elliptic curves, degree one polarized, the intersection complex of
$X_0$ is $B = \RR /\ZZ$ (the moment polytope of $(\PP^1,\O(1))$
glued at its endpoints). We take $\beta_\trop=B$. Consider the
moment map
\[
\PP^1 \lra [0,1],\quad (z:w)\longmapsto \frac{|z|^2}{|z|^2+|w|^2}
\]
Identifying endpoints in source and target respectively gives a
continuous map $\pi:X_0\ra B$ sending the node to $\{0\}$. The
fibres away from $\{0\}$ are circles. Let $\beta_0\subset X_0$ be
the lift of $\beta_{\trop}$ to $X_0$, i.e. a section of $\pi$ that
maps to the non-negative real locus $\{(s:t)\mid s,t\in\RR_{\ge
0}\}\subset\PP^1$. We want to lift $\beta_0$ further to the nearby
fibres $X_t$ under a retraction map $r_t:X_t\subset\shX\ra X_0$ to a
cycle $\beta$ (in our current example $\beta$ is going to be
homeomorphic to $\beta_0$). Restricting $\pi\circ r_t$ to $\beta$, 
defines a projection $\pi_\beta:\beta\ra\beta_\trop$ (here a
homeomorphism). We then compute the function $g(t)=\int_\beta\Omega$
on $\DD$ by patching $\beta$ via various open charts
$\pi_\beta^{-1}(W)$, $W\in\shW$, with $W\subset \beta_\trop$ such
that $g(t)$ decomposes as a sum of holomorphic functions
\[
g(t)=\sum_{W\in\shW} g_W(t),\quad\qquad
g_W(t)=\int_{\pi_\beta^{-1}(W)\cap \beta}\Omega.
\]
Since $f$ is smooth along $X_0$ away from the node, there exist
$0<\eps,\eps'<1$ so that for 
\[
V_1=\{(z,w)\in\PP^1\mid \eps<\frac{z}w<\eps'{}^{-1}\} =
\pi^{-1}(W_1),\qquad
W_1=\left(\frac{\eps^2}{\eps^2+1},\frac1{1+(\eps')^2}\right) 
\] 
and a smaller disk $\DD'\subset\DD$ we find an embedding of
$U_1:=V_1\times\DD'$ in $\shX$ such that $f|_{V_1\times\DD'}:
V_1\times\DD'\lra \DD'$ is the second projection and we may assume
(by modifying the embedding if necessary) that the retraction is the
first projection $r:V_1\times\DD'\ra V_1$. Let
$\hat\eps',\hat\eps'{}^{-1}\in \RR_{>0}$ denote the two points of
intersection of $\beta_0$ with the boundary of $V_1$. We have that
\[
g_1(t) =\int_{\beta\cap (\pi\circ r_t)^{-1}(W_1)}\Omega =
\int_{\hat\eps}^{\hat\eps'{}^{-1}} \frac{du}u = -\log
{\hat\eps'}-\log{\hat\eps}
\]
does not depend on $t$. We set 
\[
U_2 = \{(z,w,t)\in U\mid |z|<\eps, |w|<\eps'\}
\]
\begin{figure}
\resizebox{0.6\textwidth}{!}{
\includegraphics{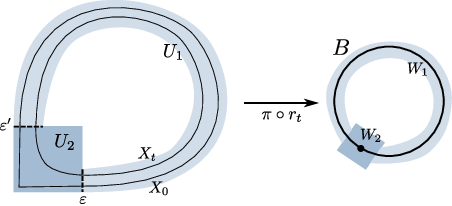}}
\caption{The covering of the degeneration of the Tate curve that is
used to decompose the period integral, pand the rojection to the
intersection complex $B$}
\label{Fig: cover-ell-curve}
\end{figure}
Let $r:U_2\ra V_2:=\tilde W_2\cap X_0$ be a retraction that
coincides at the ends $|z|=\eps$ and $|w|=\eps'$ with the retraction
induced from $U_1\ra V_1$. Let $r_t$ be its restriction to $X_t\cap
U_2$. We compute
\[
g_2(t) =\int_{r_t^{-1}(\beta_0)}\Omega =
\int^{\hat\eps}_{{\hat\eps'{}^{-1}}t^k} \frac{dz}z = \log
{\hat\eps'}+\log{\hat\eps}-\log t^k
\]
where we used $z=w^{-1} t^k$.
We conclude
\[
g(t) = g_1(t)+g_2(t) = -\log t^k
\]
which coincides with \eqref{intbeta}. The patching method is
certainly unnecessarily complicated for the Tate curve but it
illustrates the approach that will generalize to higher dimensions.

While the Tate curve demonstrates some key features of our period
calculation already, there are the following aspects that we
additionally need to consider in higher dimensions.
\begin{enumerate}
\item
The local model at a singular point of $X_0$ met by $\beta_0$ more
generally takes the shape
\[
zw=ft^\kappa .
\]
We show that $f$ basically can be assumed to equal $1$ as it does
not contribute to $\int_\beta\Omega$. This is remarkable because the
so-called slab functions $f$ are known to carry enumerative
information \cite{GHK},\cite{ICM},\cite{lau}.
As our proof shows, it is precisely the normalization condition that
determines the relevant enumerative corrections necessary to make
the mirror map trivial. We certainly expect $f$ to enter the
calculation of periods of higher weight.
\item
While the tropical cycle in the dual intersection complex
$\beta_{\trop}\subset B$ remains one-dimensional, its lift to $X_t$
will be $n$-dimensional for $n=\dim X_t$.  The projection
$\pi_\beta=(\pi\circ r)|_\beta:\beta\ra \beta_{\trop}$ will
generically be a $T^{n-1}=(S^1)^{n-1}$ fibration.  In order to pick
$T^{n-1}$ among various choices in the fibres of $\pi\circ r$, we
decorate $\beta_\trop$ with a section $\xi$ of the local system of
integral flat tangent vectors $\Lambda$ on the smooth part of $B$,
see Definition~\ref{def-trop-1-cycle}.
\item
The tropical $1$-cycle $\beta_{\trop}$ will typically have tropical
features, i.e. it is not necessarily just an $S^1$ as above but may
bifurcate satisfying a balancing condition \eqref{balancing}.
\item
Some effort is necessary to show that the cycles coming from
tropical $1$-cycles generate all cycles in the graded piece of the
monodromy weight filtration responsible for the flat coordinates,
see Corollary~\ref{genW2W0}.  Besides what is said there, we prove a
general homology-cohomology comparison theorem for (co-)homology
with coefficients in a constructible sheaf in
\S\ref{sec-hocoho-compare} as well as a comparison of simplicial and
usual sheaf homology in \S\ref{general-constructible-sheaves}.
Theorem~\ref{enough-cycles} is essentially a corollary of this.
\end{enumerate}

%===========================================================
%
%		Finite order analytic extensions and general setup
%
%===========================================================

\section{Analytic extensions and general setup}

\subsection{Analytic extensions in the compact case}
\label{Subsect: Analytic extension}
The canonical smoothing obtained from \cite{affinecomplex} is a
formal family $X$ over $\Spec\CC\lfor t\rfor$. Since the periods $g_\beta
=\int_\beta\Omega$ have essential singularities at $0$, a word is
due on how we compute these using the finite order thickenings $X_k$
of $X_0$. If $X_0$ is non-compact, we need to make the assumption
that for any $k$ there is an analytic space $\shX$ with a
holomorphic map to a disc $\Tbs$ such that its base change to
$\Spec\CC[t]/t^{k+1}$ is isomorphic to $X_k$. We call such an
$\shX\ra T$ an \emph{analytic extension} of $X_k$. If $X_0$ is
compact, an analytic extension of $X_k$ is obtained from the
following result.
\begin{theorem} (\cite[Th\'eor\`eme principal, p.598]{Do74};
\cite[Hauptsatz, p.140]{Gr74})
\label{Thm: versal}
Let $X_0$ be a compact complex-analytic space. Then there exists a
proper and flat map $\pi:\tilde\shX\to S$ of complex-analytic spaces
and a point $O\in S$ together with an isomorphism $\pi^{-1}(O)\simeq
X_0$ such that $\pi$ is versal at $O$ in the category of complex
analytic spaces.
\end{theorem}

Indeed, by Theorem~\ref{Thm: versal} the formal family $\limdir X_k$
is obtained by pull-back of the versal deformation $\tilde\shX\to S$
of $X_0$ by a formal arc in $S$. Such a formal arc can be
approximated to arbitrarily high order by a map from a holomorphic
disc $\Tbs$ to $S$, and $\shX\to \Tbs$ is then defined by pull-back
of the versal family.

By the definition of $\Omega$, $\int_\alpha\Omega$ is constant on $\Tbs$. 
Furthermore, 
\[
h_\beta = \exp\left(-2\pi i \frac{\int_\beta \Omega}{\int_\alpha
\Omega}\right)
\] 
is going to be a holomorphic function on $\Tbs$, so $h_\beta$ is
determined by its power series expansion at $0$. The Taylor
series of this function up to order $k$ is determined by $X_k$ and
hence does not depend on the choice of $\shX$.  This is true for any
$k$, so we obtain in this way the entire Taylor series of $h_\beta$
at $0$ independent of the choices of $\shX$. We will see that
this is the Taylor series of a holomorphic function.

Furthermore, for each $k$ and each analytic extension $\shX$ of
$X_k$, we will consider a collection $\shU$ of pairwise disjoint
open sets in $\shX$ whose closures cover $\shX$ with zero measure
boundary. We then decompose
\[
h_\beta=\prod_{U\in \shU} h_{U}
\] 
where $h_{U}=\exp \big(-2\pi i \frac{\int_{\beta\cap U}
\Omega}{\int_\alpha \Omega}\big)$. We will choose the open sets such
that $h_U$ is holomorphic for each $U\in\shU$. Let $U_k$ denote the
base change of $U$ to $\Spec\CC[t]/t^{k+1}$. Also, the $k$th
$t$-order cut-off $h^k_\beta$ of $h_\beta$ decomposes
$h^k_\beta=\prod_{U} h^k_{U}$  in the $k$th order cut-offs $h^k_{U}$
of the $h_U$.  Hence we can compute each $h^k_\beta$ from an open
cover like $\shU$.

We next remind ourselves of the relevant notions developed in
\cite{logmirror1,affinecomplex}.

%===========================================================

\subsection{Toric degenerations and log CY spaces} 
The full definition of a \emph{toric degeneration} can be found in
\cite[Definition 1.8]{affinecomplex}.  Most importantly, it is a
flat morphism $f:\shX\ra\Tbs$ with the following properties:
\begin{enumerate}
\item
$\shX$ is normal,
\item
$\Tbs=\Spec R$ for $R$ a discrete valuation $\CC\lfor t\rfor$-algebra,
\item
the normalization $\tilde X_0$ of the central fibre $X_0$ is a
union of toric varieties glued torically along boundary strata such
that
\item
away from a locus $\shZ\subset\shX$ of relative codimension two,
the triple $(\shX,X_0,f)$ is locally given by $(U,V,z^\rho)$ with
$U$ an affine toric variety with reduced toric divisor $V$ cut out
by a monomial $z^\rho$,
\item
$\shZ$ is required not to contain any toric strata of $X_0$, \item
the normalization $\tilde X_0\ra X_0$ is required to be $2:1$ on the
union of the toric divisors of $X_0$ except for a divisor $\tilde
D\subset \tilde X_0$ where it may be generically $1:1$,
\item
denoting by $D$ the image of $\tilde D$ in $X_0$, the local model
$(U,V,z^\rho)$ at a point of $D\setminus\shZ$ can be chosen so that
$D$ is defined by $(z^{\rho_D},z^\rho)$ for $z^{\rho_D}$ another
monomial,
\item
the components of $X_0$ are algebraically convex, i.e. they admit a
proper map to an affine variety.
\end{enumerate}
One similarly defines a \emph{formal toric degeneration} as a family
over $\Spf\CC\lfor t\rfor$. A \emph{polarization} of a toric degeneration
is a fibre-wise ample line bundle. At a generic point $\eta_\tau$ of
a stratum $X_\tau$ of $X_0$, let $P_\tau$ denote the toric monoid
such that 
\begin{equation}
\label{UPtau}
U=\Spec\CC[P_\tau]
\end{equation}
for $(U,V,z^\rho)$ the local model at $\eta_\tau$ (which exists
because $\shZ$ does not contain $\eta_\tau$ by (5)). One finds that
$P_\tau$ is unique if one requires $\rho$ (respectively
$\rho+\rho_D$ at a point in $D$) to be contained in its relative
interior of $P_\tau$ which we assume from now on. Even though we do
not use any log geometry in this paper, we should mention that the
data of the local models $(U,V,z^\rho)$ can be elegantly encoded in
a \emph{log structure} on $X_0$. This is a sheaf of monoids
$\shM_{X_0}$ on $X_0$ together with a map of monoids
$\alpha:\shM_{X_0}\ra\shO_{X_0}$ using the multiplication on
$\shO_{X_0}$.  It is required that the structure map $\alpha$
induces an isomorphism
$\alpha^{-1}(\O_{X_0}^\times)\ra\shO_{X_0}^\times$. The way in which
$\shM_{X_0}$ encodes the local models is then
\begin{equation}
\label{MPtau}
\shM_{{X_0},\eta_\tau} / \alpha^{-1}(\O_{{X_0},\eta_\tau}^\times)
\oplus\ZZ^{\dim\tau} \cong P_\tau
\end{equation}
at the generic point $\eta_\tau$ of the stratum $X_\tau$ not
contained in $D$ and there is a similar relation on $D$. The
isomorphism \eqref{MPtau} is not canonical unless $\dim\tau=0$. Also
the monomial $z^\rho$ is encoded in the log structure as it is part
of the data of the log morphism from $X_0$ to the standard log
point. One defines a \emph{toric log CY-pair} to be a space $X_0$
with log structure $\shM_{X_0}$ satisfying a list of criteria that
is induced by the list above on the central fibre $X_0$, see
\cite[Definition 1.6]{affinecomplex}.

%===========================================================

\subsection{Intersection complex}
\label{subsec-intcomplex}
We recall \cite[\S4.2]{logmirror1}. Let $X_0$ denote the central
fibre of a polarized toric degeneration (in fact a pre-polarization
suffices, see \cite[Ex. 1.13]{affinecomplex}). By affine convexity
and the polarization, each irreducible component of $X_0$ is a toric
variety $X_\sigma$ given by a lattice polyhedron $\sigma$. We glue
two maximal polyhedra $\sigma_1,\sigma_2$ along a facet $\tau$ if
$\tau$ corresponds to a divisor in the intersection of
$X_{\sigma_1}$ and $X_{\sigma_2}$. The resulting space $B$ of all
such gluings is a topological manifold. Let $\P$ denote the set of
polyhedra and their faces modulo identifications by gluing. To each
cell $\tau\in\P$ corresponds a stratum $X_\tau$ of $X_0$ and this
association is compatible with inclusions and dimensions
($\dim\tau=\dim X_\tau$). We will denote by $\P^{[k]}$ the subset of
$k$-dimensional faces and by abuse of notation sometimes also their
union in $B$. Since $\shZ$ does not contain any toric strata,  at
the generic point of a stratum $X_\tau$ of $X_0$ there is a toric
local model $(U,V,z^\rho)$ and $U=\Spec\CC[P_\tau]$ with $\rho\in
P_\tau$. The monoid $P_\tau$ embeds in its associated group
$P_\tau^\gp\cong \ZZ^{n+1}$. Let $P_{\tau,\RR}$ denote the convex
hull of $P_\tau$ in $P_{\tau,\RR}^\gp=P_\tau^\gp\otimes_\ZZ\RR$. 
Let now $\tau=v$ be a vertex. If $X_{\sigma_1},\ldots,X_{\sigma_r}$
are the $n$-dimensional strata containing the point $X_v$ then
$\sigma_1,\ldots,\sigma_r$ correspond to facets of $P_{v,\RR}$.
The composition of the embedding of the facets with the projection 
\begin{equation} \label{projectPv}
P_{v,\RR}\lra P^\gp_{v,\RR}/\RR\rho \cong \RR^n
\end{equation}
provides a chart of the topological manifold $B$ in a neighbourhood
of $v$. Together with the relative interiors of the maximal cells of
$\P$, the charts provide an integral affine structure on $B$ away
from a codimension two locus $\Delta$. Thus there is an atlas for
$B\setminus\Delta$ with transition functions in
$\GL_n(\ZZ)\ltimes\ZZ^n$. The singular locus $\Delta$ can be chosen
to be contained in the union of those simplices in the barycentric
subdivision of $\P$ that do not contain a vertex or barycenter of a
maximal cell, see \cite[Remark 1.49]{logmirror1}.\footnote{The
precise choice of $\Delta$ is irrelevant for the present paper
because all our computations are localized near $\beta_\trop$ which
is chosen disjoint from $\Delta$. For example, in \cite{theta}
$\Delta$ is enlargeded to contain all codimension two cells of the
barycentric subdivsion of $\P$ that are not intersecting the
interiors of maximal cells.}
The pair $(B,\P)$
is called the \emph{intersection complex} of $X_0$ (or of $\shX$). 
The local models $(P_\tau,\rho)$ provide an additional datum not yet
captured in $(B,\P)$. This is a strictly convex multivalued
piecewise affine function $\varphi$ on $B$, that is, a collection of
continuous functions on an open cover of $B$ that are strictly
convex with respect to the polyhedral decomposition $\P$ and which
differ by affine functions on overlaps. In particular, there is a
piecewise linear representative $\varphi_v$ in a neighbourhood of
each vertex on $\P$. Let $P_{v,\RR}\cong
\Lambda_{v,\RR}\oplus\RR\rho$ be a splitting coming from an integral
section of \eqref{projectPv}. Then the boundary $\partial
P_{v,\RR}$ of $P_{v,\RR}$ gives the graph of a piecewise linear
function 
\[
\varphi_v:\Lambda_{v,\RR}\lra \RR\rho,
\] 
uniquely defined up to adding a linear function (change of
splitting). If $v\in\partial B$, then $\varphi_v$ is in fact
defined only on part of $\Lambda_v$. The collection of $P_{v}$
determines $\varphi$ completely via the collection of $\varphi_{v}$.
Conversely, we can obtain all $P_{v}$ from knowing $\varphi$ via
\[
 P_v = \{(m,a)\in\Lambda_v\oplus\ZZ\mid \varphi_v(m)\le a\}  
\]
where we identified $\rho=(0,1)$. The triple $(B,\P,\varphi)$ is
called a \emph{polarized tropical manifold}. The local model
$P_\tau$ for $\tau\in\P$ is a localization in a face corresponding
to $\tau$ of the monoid $P_v$ for any $v\in\tau$.

%===========================================================

\subsection{Simplicity}
\label{subsec-simplicity}
The intersection complex $(B,\P)$ or simply the affine manifold $B$
is called \emph{simple} if certain polytopes constructed locally
from the monodromy around $\Delta$ are elementary lattice simplices.
This condition should be viewed as a local rigidity statement.
For the precise definition, see \cite[Definition 1.60]{logmirror1}. 
It is believed that under suitable conditions, starting with an
intersection complex $(B,\P,\varphi)$ one can subdivide it to turn
it into a simple $(B,\P)$. Geometrically, this would correspond to a
further degeneration of $X_0$. This was shown to be true for all
toric degenerations arising from Batyrev-Borisov examples
\cite{Gr05}. Mumford's degenerations of abelian varieties are
automatically simple since $\Delta=\emptyset$ in this case. Hence,
simplicity is a reasonable condition. We made use of it in
Corollary~\ref{genW2W0}.

%===========================================================

\subsection{Gluing data and reconstruction}
\label{subsec-gluing}
By the main result of \cite{affinecomplex}, any toric log CY space
$X_0^\dagger$ with simple dual intersection complex is the central
fibre of a formal toric degeneration $f:\shX\ra\Spf\CC\lfor t\rfor$. 
Furthermore, given $X_0^\ls$ there is a canonical such toric
degeneration. One constructs this order by order, so for any $k$, a
map $f_k:X_k\ra \Spec\CC[t]/ t^{k+1}$ is built such that the
collection of these is compatible under restriction. We follow
\cite[\S2.2, \S5.2]{theta}. Let $X^\circ_0$ denote the complement of
all codimension two strata in $X_0$. It suffices to produce a
smoothing of $X^\circ_0$ by a similar collection of finite order
thickenings $X^\circ_k$, see \cite{theta}. We can cover $X^\circ_0$
with two kinds of charts, $U_\sigma$ for $\sigma\in\P$ a maximal cell
and $U_\rho$ with $\rho\in\P$ of codimension one. For a maximal cell
$\sigma\in\P$, we denote by $\Lambda_\sigma$ the stalk of $\Lambda$ at
a point in the relative interior of $\sigma$ (any two choices are
canonically identified by parallel transport). For a codimension one
cell $\rho\in\P^{[n-1]}$, we denote by $\Lambda_\rho$ the tangent
lattice to $\rho$. This is invariant under local monodromy and thus
also independent of a stalk of $\Lambda$ in $\Int\rho$. The two types
of open sets are now given by $U_\sigma=\Spec\CC[\Lambda_\sigma]$ and
$U_\rho=\Spec\CC[\Lambda_\rho] [Z_+,Z_-]/(Z_+Z_-)$ for
$\sigma\in\P^{[n]}$ and $\rho\in\P^{[n-1]}$ respectively. We will give
the transitions for these and simultaneously that for their $k$th
order thickenings. We fix a maximal cell $\sigma=\sigma(\rho)$ for
each $\rho\in\P^{[n-1]}$. We also fix a tangent vector $w=w(\rho)\in
\Lambda_\sigma$ such that
\[
\Lambda_\sigma = \Lambda_\rho+\ZZ w.
\]
So $w$ projects to a generator of $\Lambda_\sigma/ \Lambda_\rho
\cong\ZZ$ and we choose it so that it points from $\rho$ into
$\sigma$. Assume we are given for each $\ul\rho\in\tilde\P^{[n-1]}$
a polynomial $f_{\ul\rho}\in \CC[\Lambda_\rho]$ with the following
compatibility property. Let $\sigma'$ denote the other maximal cell
besides $\sigma$ that contains $\rho$. Let $\ul\rho, \ul\rho'
\in\tilde\P^{[n-1]}$ be both contained in $\rho\in\P^{[n-1]}$. The
monodromy along a loop that starts in $\sigma$, passes via $\ul\rho$
into $\sigma'$ and via $\ul\rho'$ back into $\sigma$ is given by an
automorphism of $\Lambda_\sigma$ that fixes $\Lambda_\rho$ and maps
\[
w\longmapsto w+m_{\ul\rho\ul\rho'}
\] 
for some $m_{\ul\rho\ul\rho'}\in\Lambda_\rho$.
The required compatibility between $f_{\ul\rho}$ and $f_{\ul\rho'}$ is then
\begin{equation}
\label{compat-frho}
t^{\kappa_{\ul\rho}}f_{\ul\rho}=z^{m_{\ul\rho\ul\rho'}}
t^{\kappa_{\ul\rho'}}f_{\ul\rho'}
\end{equation}
so $f_{\ul\rho}$ determines $f_{\ul\rho'}$ uniquely and vice versa. We
give thickened versions of $U_\sigma$ and $U_\rho$ by giving the
corresponding rings as
\[
R^k_\sigma = A_k[\Lambda_\sigma],
\]
\[
R^k_{\ul\rho} = A_k[\Lambda_\rho][Z_+,Z_-]/(Z_+Z_--f_{\ul\rho}\cdot
t^{\kappa_{\ul\rho}})
\]
for $A_k=\CC[t]/t^{k+1}$. When $k$ is fixed, we also write $R_\sigma$
for $R^k_\sigma$ and so forth. For $\sigma=\sigma(\rho)$, the map 
\[
\chi^\can_{\ul\rho,\sigma}:R_{\ul\rho} \lra R_\sigma
\] 
is isomorphic to the localization map 
\[
R_{\ul\rho} \lra (R_{\ul\rho})_{Z_+}
\]
by identifying $R_\sigma=(R_{\ul\rho})_{Z_+}$ via $Z_+=z^w$ and
elimination of $Z_-$ via $Z_-=Z_+^{-1} f_{\ul\rho}
t^{\kappa_{\ul\rho}}$. Similarly, if $\sigma'$ is the other maximal
cell containing $\rho$ then we obtain a vector $w_{\ul\rho}\in
\Lambda_{\sigma'}$ by parallel transporting $w$ from $\sigma$ to
$\sigma'$ via $\ul\rho$. The map $R_{\ul\rho} \ra R_{\sigma'}$ is then
given by identifying $R_{\sigma'}= (R_{\ul\rho})_{Z_-}$ via
$Z_-=z^{-w_{\ul\rho}}$ and elimination of $Z_+$. The compatibility
condition \eqref{compat-frho} implies that we have canonical
isomorphisms $R_{\ul\rho}\cong R_{\ul\rho'}$ compatible with the maps
to $R_\sigma$ and $R_\sigma'$, namely $R_{\ul\rho}\ra
R_{\ul\rho'}$ via
\[
Z_+\longmapsto Z_+,
\]
\[
Z_-\longmapsto Z_-z^{m_{\ul\rho\ul\rho'}}.
\]

\begin{definition} 
\label{opengluingdata}
(Open) gluing data $(s_{\ul\rho,\sigma})$ is a collection of
homomorphisms $s_{\ul\rho,\sigma}:\Lambda_\sigma\ra A_0^\times$, one
for each pair $\ul\rho\subset\sigma$.
\end{definition}

We can twist the maps $\chi^\can_{\ul\rho,\sigma}$ to a map
$\chi_{\ul\rho,\sigma}$ by composing $\chi^\can_{\ul\rho,\sigma}$ with
the automorphism 
\[
R_\sigma\lra R_\sigma,\qquad z^m\longmapsto s_{\ul\rho,\sigma}(m) z^m.
\] 

While it is possible to glue all charts to a scheme we want to modify
the construction further before starting the gluing,
cf.~\cite[\S2.3]{theta}. In order to obtain the $X_k$ from
\cite{affinecomplex}, we need to consider certain combinatorial data
that gives the rings of the charts that glue to $X_k$. For the rings,
this will simply be taking further copies of the $R_\sigma$ and
$R_{\ul\rho}$ that we defined already but there will be new maps and
the rings get modified as we need to add higher order terms to the
$f_{\ul\rho}$. The combinatorial object determining $X_k$ is called a
\emph{structure} whose data we now describe.  It comes with a
refinement $\P_k$ of $\P$.  The maximal cells $\fou\in\P_k^{[n]}$
are called \emph{chambers}.  We denote by $\sigma_\fou\in\P^{[n]}$
the unique maximal cell in $\P$ containing $\fou$. The codimension
one cells of $\P_k$ are called \emph{walls} and are denoted $\fop$. A
wall that is contained in $\P^{[n-1]}$ is called a \emph{slab} and it
is in fact contained in a unique $\ul\rho_{\fop}
\in\tilde\P^{[n-1]}$. We typically denote slabs by $\fob$. A wall
that is not a slab is called a \emph{proper wall} and we then denote
by $\sigma_{\fop}\in \P^{[n]}$ the unique maximal cell containing
$\fop$. Each wall $\fop$ comes with a polynomial
$f_\fop\in A_k[\Lambda_{\fop}]$ and if $\fop$ is a slab
and $k=1$ then $f_\fop=f_{\ul\rho_\fop}$. If $\fop$ is a
proper wall then 
\begin{equation}
f_\fop \equiv 1\mod t,
\end{equation}
so $f_\fop$ is invertible.

The rings that give the charts to glue to $X_k^\circ$ are then derived
from the known rings by
\begin{align}
R^k_{\fou}&:=R^k_{\sigma_{\fou}}=A_k[\Lambda_{\sigma_{\fou}}]
&\qquad\hbox{ for a chamber }\fou,\\
R^k_{\fop}&:=R^k_{\sigma_{\fop}}=A_k[\Lambda_{\sigma_{\fop}}]
&\qquad\hbox{ for a proper wall }\fop,\\
R^k_{\fob}&:=A_k[\Lambda_{\fob}][Z_+,Z_-]/(Z_+Z_--f_{\fob}
\cdot t^{\kappa_{\ul\rho_{\fob}}}) &\qquad\hbox{ for a slab }\fob.
\end{align}
Once $k$ is fixed, we will write $R_\fou$ for $R^k_\fou$ and
so forth. For every pair of a slab $\fob$ contained in a chamber
$\fou$ we have a map 
\[
\chi_{\fob,\fou} :R_\fob \lra R_\fou
\] 
defined just like $\chi_{\ul\rho,\sigma}$; in particular,
$\chi_{\fob,\fou}$ incorporates the twist by the gluing automorphism
\[
s_{\ul\rho_{\fob},\sigma_{\fou}}:\Lambda_{\fou}\lra A_0^\times.
\]
Furthermore, when $\fou,\fou'$ are two chambers joined by a
proper wall $\fop$, we have an isomorphism
\[
\theta_\fop: R_{\fou}\lra R_{\fou'}, \qquad z^m\longmapsto
f_\fop^{\langle d_\fop, m\rangle} z^m
\]
where $d_\fop$ is the generator of
$\Lambda_\fop^\perp\subset\check\Lambda_{\sigma_\fop}$ that
points from $\fop$ into $\fou$.

Using these chart transitions, one glues a scheme $X^\circ_k$ for
every $k$ and these are compatible with restrictions $X^\circ_k\ra
X^\circ_{k-1}$.

%===========================================================

\subsection{The normalization condition}
\label{subsect: normalization condition}
The algorithm of \cite{affinecomplex} produces the $(k+1)$-structure
data $\P_{k+1}$ and all $f_\fop$ from the analogous $k$-structure
data. It is inductive in $k$ and deterministic. The $f_\fop$ will
get modified by higher order terms in $t$ that are being added.
Apart from a modification that comes from what is called
\emph{scattering} and that we do not need to go into here, there is
another crucial step to guarantee canonicity which is the
\emph{normalization condition}. The function $f_{\fop}$ associated
to a wall is an element of $A_k[\Lambda_\fop]$ with non-vanishing
constant term $a\in\CC$. If $\fop$ is a proper wall then $a=1$. For
$\fop=\fob$ a slab we consider the condition, in a suitable
completion of $A_k[\Lambda_\fop]$
\cite[Construction~3.24]{affinecomplex},
\begin{equation}
\label{ncond}
\log\left(\frac{f_\fob}{a}\right)=\sum_{i\ge
1}\frac{(-1)^{i+1}}{i}\left(\frac{f_\fop}{a}-1\right)^i \hbox{
has no monomial of the form $t^e$ with $1\le e\le k$}.
\end{equation}
The condition \eqref{ncond} is called the \emph{normalization up to
order $k$} and can be found in \cite[III. Normalization of
slabs]{affinecomplex}.  It becomes a crucial ingredient for our main
result as every $f_\fop$ in the $k$-structure is $k$-normalized
by loc.cit..

%===========================================================
%
%		From tropical cycles to homology cycles
%
%===========================================================

\section{From tropical cycles to homology cycles}
\label{section-tropical-to-homology}
We assume that $B$ is oriented. Assume that we are given a tropical
cycle $\beta_\trop$ on $B$ and the data of the order $k$ structure.
As explained in \S\ref{subsec-gluing} from this data we can glue the
scheme $X^\circ_k$. We assume $X_k^\circ$ partially compactifies to
$X_k$, a flat deformation of $X_0$. This can be assured with
additional consistency assumptions that come out from the
construction in \cite{affinecomplex}, or in the quasiprojective case
can be formulated via theta functions \cite{theta}. We furthermore
assume the existence of an analytic extension $f:X\ra\DD$ of
$f_k:X_k\ra\Spec A_k$, as established for the compact case in
\S\ref{Subsect: Analytic extension}. By perturbing $\beta_\trop$ if
necessary, we assume that its vertices lie in the interior of the
chambers of $\P_k$ and that the edges meet the walls in finitely
many points in their interiors. Let $\shW$ be a collection of
disjoint open sets of
\begin{figure}
\resizebox{0.6\textwidth}{!}{
\includegraphics{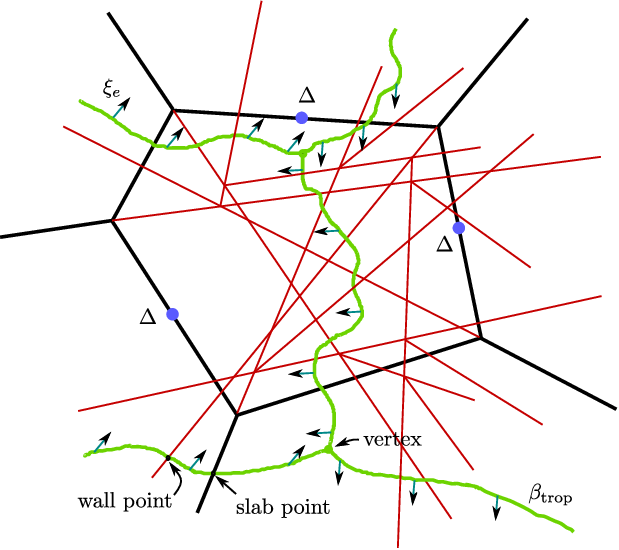}
}
\caption{Types of points in $\beta_\trop$}
\label{Fig: types-of-points}
\end{figure}
$B$ such that the union of their closures covers $\beta_\trop$. We
may assume that each point in $\beta_\trop$ which either is a vertex
or lies in a wall of $\P_k$, is contained in a $W\in\shW$ and each
$W$ contains at most one such point. Furthermore we assume that the
closures of any two open sets that contain a wall point are
disjoint. We have thus four types of open sets $W\in\shW$ given by
whether it contains a vertex, slab point, wall point or none of
these. For a chamber $\fou$, set $\sigma=\sigma_\fou$ and
consider the moment map 
\[
\mu_\sigma:\Hom(\Lambda_\sigma,\CC^\times)\ra
\Int\sigma, \qquad (z_1,\ldots,z_n)\longmapsto
\frac{\sum_{m\in\sigma\cap \Lambda_\sigma} |z^m(z_1,\ldots,z_n)|^2
\cdot m}{\sum_{m\in\sigma\cap \Lambda_\sigma}|z^m(z_1,\ldots,z_n)|^2}.
\]
For the following discussion we identify $\CC^\times= \RR_{>0}\times
S^1$ as real Lie groups, via absolute value and argument. Recall
that $\mu_\sigma$ identifies cotangent vectors of $\sigma$ with
algebraic vector fields on $\Hom(\Lambda_\sigma, \CC^\times)\simeq
(\CC^*)^n$. The induced action of $\Hom(\Lambda_\sigma,S^1)\simeq (S^1)^n$
acts simply transitively on the fibres of $\mu_\sigma$. Moreover,
there is a canonical section $S:\Int\sigma\ra
\Hom(\Lambda_\sigma,\RR_{>0})\subset
\Hom(\Lambda_\sigma,\CC^\times)$. We use this section in the
interior of each chamber $\fou$ to lift $\beta_\trop\cap
\fou$ to $\Spec\CC[\Lambda_\fou]$. Now for an edge $e$ of
$\beta_\trop$, the section $\xi_e\in\Gamma(e,\Lambda)$ defines a
union of translates of a real $(n-1)$-torus
\[
T_{e}=\{\phi\in\Hom(\Lambda_\sigma,S^1)\mid \phi(\xi_e)=1\}
\subset \Hom(\Gamma(e,\Lambda),S^1)\simeq (S^1)^n.
\] 
Note that $T_e$ is connected if and only if $\xi_e$ is primitive.
Let $\bar\xi_e$ be the primitive vector such that $\xi_e$ is a
positive multiple of $\bar\xi_e$. By assumption $\Lambda$ is
oriented and $\xi_e$ induces an orientation of $T_{e}$, namely a
basis $v_2,\ldots,v_n$ of $\Lambda_e/\ZZ\bar\xi_e$ is oriented if 
$\bar\xi_e,v_2,\ldots,v_n$ is an oriented basis of $\Lambda_e$.
Identifying $\Gamma(e,\Lambda)$ with $\Lambda_\sigma$ when $e$
passes through $\sigma$, we obtain a subgroup $T_e\subset
\Hom(\Lambda_\sigma,S^1)$. Over the interior of a chamber
$\fou$ we consider the orbit of the section $S$ under this torus
which we denote by 
\[
T_{e,\fou}:=\Hom(\xi_e^\perp,S^1)\cdot S(e\cap\Int\fou).
\]
Topologically $T_{e,\fou}$ is homeomorphic to a cylinder, a product
of $(S^1)^{n-1}$ with a closed interval.
Each cylinder $T_{e,\fou}$
receives an orientation from the orientation of $T_e$ and the
orientation of $e$. We wish to connect these cylinders to a cycle
$\beta$ and we want to transport them to the nearby fibres. 

%===========================================================

\subsection{The open cover}
We construct a collection $\shU$ of disjoint open sets of $\shX$
whose closures will cover the lift $\beta$ of $\beta_\trop$ to all
nearby fibres. The family $f:\shX\ra\DD$ is smooth away from the
codimension one strata of $X_0$. For each chamber $\fou$, we have a
canonical projection $\Spec A_k[\Lambda_\fou]\ra
\Spec\CC[\Lambda_\fou]$ and a momentum map $\mu=\mu_{\sigma_\fou}$.
If $W$ is an open set entirely contained in $\fou$, then we include
in $\shU$ a corresponding open set $U$ whose intersection with $X_0$
is $\mu^{-1}(W)$ and where the family $f$ is trivialized,
that is, $U=\mu^{-1}(W)\times\DD$ and $f$ is the projection
to $\DD$. Furthermore, we may ask for the other projection
$U\to\mu^{-1}(W)$ to coincide with the restriction of $\Spec
A_k[\Lambda_\fou]\ra \Spec \CC[\Lambda_\fou]$ to the inverse image
of $\mu^{-1}(W)$. 

If $W$ meets a proper wall $\fop$ we have two chambers $\fou$,
$\fou'$ containing $\fop$. Again, we want $U$ to have the property
that $U\cap X_0=\mu^{-1}(W)$. We have a trivialization
$\mu^{-1}(W)\times \Spec A_k$ in the chart $\Spec A_k[\Lambda_\fou]$
and another such in the chart $\Spec A_k[\Lambda_{\fou'}]$. These
differ by $\theta_\fop$. We lift both trivializations from $\Spec
A_k$ to $\DD$ and take $U\subset\shX$ to be an open set containing
both of these. So we do not typically have a projection $U\ra
\mu^{-1}(W)$ on all of $U$. To make sure $U$ is disjoint from
the previously constructed open sets, we may remove the closures of
the other open sets from $U$ if necessary.

It remains to include an open set $U$ in $\shU$ for each $W$
containing a slab point. Let $\fob$ be the slab and $\fou,\fou'$ be
the two chambers containing it. Here we simply take any sufficiently
large open set $U$ with the closure of $U\cap X_0$ agreeing with the
closure of $\mu^{-1}(W\cap\Int\fou)\cup {\mu'}^{-1}
(W\cap\Int\fou')$. Here $\mu'$ is the momentum map for the maximal
cell containing $\fou'$. To make the elements of $\shU$ pairwise
disjoint we remove the closures of the just constructed open sets
that intersect the singular locus of $X_0$ from the previously
constructed sets in $\shU$. This way, we obtain the desired
collection of open sets.

Note that except for the open sets at the slabs, there is a natural
way now to transport the constructed cylinders $T_{e,\fou}$ to every
nearby fibre $f^{-1}(t)$ for $t\in \DD$ by taking inverse images
under the projections $U\ra U\cap X_0$. We next close the union of
cylinders to an $n$-cycle $\beta$.

%===========================================================

\subsection{Closing $\beta$ at vertices}
\label{sec-beta-vertex}
First consider the situation at a vertex $v\in\beta_\trop$. Let
$\fou$ be the chamber containing it.  We identify
$\mu^{-1}(v)=\Hom(\Lambda_v,S^1). S(v)$ with $\Hom(\Lambda_v,S^1)$.
This $n$-torus contains various $(n-1)$-tori
\[
T_{e,v}:=\mu^{-1}(v)\cap T_{e,\fou} = T_e\cdot S(v).
\] 
For each edge $e$ of $\beta_\trop$ that attaches to $v$ there is one
$T_{e,v}$, and if $\xi_e$ is an $m_e$-fold multiple of a primitive vector then $T_{e,v}$ is a
disjoint union of $m_e$ cylinders. Note that $T_{e,v}$ carries the
sign $\eps_{e,v}$ as the induced orientation from $T_{e,\fou}$. We
want to show that the union of the $T_{e}\cdot S(v)$ is the boundary
of an $n$-chain in the real $n$-torus $\mu^{-1}(v)$, so that we can glue
them. Since the $(n-1)$-cycle $T_{e,v}$ is identified under
Poincar\'e duality with $\eps_{e,v}\xi_e|_v \in\Lambda_v=
H^1(\mu^{-1}(v),\ZZ)$,  we see that the union of the Poincar\'e
duals is trivial in cohomology by the balancing
condition~\eqref{balancing}:
\[
0=\sum_{v\in e} \eps_{e,v} \xi_e.
\]
Hence, the union of $T_{e,v}$ is trivial in homology and so there
exists an $n$-chain $\Gamma_v\subset\mu^{-1}(v)$ whose boundary is
this union.  The chain $\Gamma_v$ is unique up to adding multiples
of $\mu^{-1}(v)$.

%===========================================================

\subsection{Closing $\beta$ at proper walls}
\label{sec-beta-wall}
Let $U\in\shU$ be an open set the corresponding open set $W\subseteq
B$ of contains a point $p$ in a proper wall $\fop$. Let $e$ be
the edge of $\beta_\trop$ containing $p$. By construction, $U$
contains $\mu^{-1}(W)\times\DD$ in two different ways given by the
trivializations in $\Spec A_k[\Lambda_\fou]$ and $\Spec
A_k[\Lambda_{\fou'}]$ respectively. We transport $T_{e,\fou}$
and $T_{e,\fou'}$ into the nearby fibres by the respective
trivializations. Then we need to connect theses transports over $p$.
We do this by interpolation along straight real lines in
$\DD\times(\CC^*)^n\subset\CC^{n+1}$. Let
$v_1,\ldots,v_n$ be an oriented basis of $\Lambda_p$ with $v_1$
pointing from $\fop$ into $\fou$ and so that
$v_2,\ldots,v_n$ is a basis of $\Lambda_\fop$. As before, we set
$z_j=z^{v_j}$. Writing $f_\fop=1+\tilde f$, we have
\[
\theta_\fop:\Spec R_{\fou'}\lra \Spec R_\fou,
(t,z_1,\ldots,z_n)\longmapsto (t,(1+\tilde
f(t,z_2,\ldots,z_n))z_1,z_2,\ldots,z_n).
\]
For fixed $x=(t,z_1,z_2,\ldots,z_n)$ consider the map
\[
\gamma_x:[0,1]\to \Spec R_\fou,\quad \lambda\longmapsto
(t,(1+\lambda \tilde f(t,z_2,\ldots,z_n))z_1,z_2,\ldots,z_n).
\]
As a map in $x$ this map extends to a neighbourhood of $U\cap X_0$,
which we denote by the same symbol. Analogous conventions are
understood at several places in the sequel.
Now we define the chain $\Gamma_p= \bigcup_{x\in T_e\cdot S(p)}
\gamma_x([0,1])$ and give it the orientation induced from that of
$T_e$ and $[0,1]$. We find that $\Gamma_p$ connects $T_{e,\fou}$
with $T_{e,\fou'}$ over $p$ with the correct orientation if $e$
traverses from $\fou$ to $\fou'$ at $p$; otherwise we take
$-\Gamma_p$ for this purpose.

%===========================================================

%===========================================================
\subsection{Closing $\beta$ at slabs}
\label{sec-beta-slab}
Let $U\in\shU$ correspond to an open set $W\subset B$ that contains
a slab point $p\in\fob$. Let $e$ be the edge of $\beta_\trop$
containing $p$. We know that $U$ contains
\[
\Spec R^k_\fob=\Spec
A_k[\Lambda_{\fob}][Z_+,Z_-]/(Z_+Z_--f_{\fob}\cdot
t^{\kappa})
\]
as a non-reduced closed subspace, where $\kappa=
\kappa_{\ul\rho_{\fob}}$. Without restriction we may assume that $U$
is given by a hypersurface of the form $\tilde Z_+\tilde Z_- \tilde
f_\fob t^{\kappa_{\ul\rho_\fob}}=0$ in an open subset of
$\CC^{n+2}$. Let $\fou,\frak{u'}$ denote the chambers
containing $\fob$ with $\fou\subset\sigma(\rho_{\fob})$, with
$\sigma(\rho_\fob)$ the chosen maximal cell containing $\rho_\fob$.
So $w=w(\rho_\fob)$ points into $\fou$ and $Z_+=z^w$. As before, let
$v_1,\ldots,v_n$ be an oriented basis of $\Lambda_p$ with $v_1=w$
and $v_2,\ldots,v_n$ spanning $\Lambda_\fob$. Then
$z_i=z^{v_i}$ and $t$ give coordinates on $U\setminus (X_0\cap
U)$. These correspond to the trivialization of an open subset of
$f:U\ra\DD$ given by $\fou$. We obtain another set of coordinates
replacing $Z_+$ by $Z^{-1}_-$ that comes from the trivialization of
part of $U$ via $\fou'$. As in \S\ref{subsect: normalization
condition} we write $f_\fob=a(1+\tilde f)$ for $a\in\CC^\times$ and
$\tilde f$ having no constant term. Recall the gluing data $s_{\ul
\rho_{\fob},\sigma_{\fou}}:\Lambda_\fou\ra A_0^\times$,
$s_{\fou'}=s_{\ul\rho_{\fob}, \sigma_{\fou'}}: \Lambda_{\fou'}\ra
A_0^\times$. We identify $\Lambda_{\fou'}$ and $\Lambda_{\fou}$ by
parallel transport through $\fob$ and set 
\[
s_j=s_{\ul\rho_{\fob},\sigma_\fou}(v_j),\qquad
s_j'=s_{\ul\rho_{\fob},\sigma_{\fou'}}(v_j).
\]
If $v_1^*,\ldots,v_n^*$ denotes the dual basis to $v_1,\ldots.,v_n$
then we have
\begin{equation}
\label{s-in-terms-of-sj}
s_{\ul\rho_{\fob},\sigma_{\fou}}=\prod_{j=1}^n s_j^{v_j^*}
,\qquad   s'_{\ul\rho_{\fob},\sigma_{\fou'}}=\prod_{j=1}^n
(s'_j)^{v_j^*} .
\end{equation}
Solving for $Z_+$ yields $Z_+=f_{\fob}t^{\kappa} Z_-^{-1}$, so
the transformation on points given in the second set of coordinates
to the first reads
\[
\theta_\fob:(t,z_1,\ldots,z_n)\longmapsto
\left(t,\frac{s'_1}{s_1} a(1+\tilde f(t,s'_2 z_2,\ldots,s'_n
z_n))t^\kappa
z_1,\frac{s'_2}{s_2}z_2,\ldots,\frac{s'_n}{s_n}z_n\right).
\]
We factor $\theta_\fob= \theta_{s,a}\circ \theta_f$ where
\[
\theta_f :(t,z_1,\ldots,z_n)\longmapsto \left(t,(1+\tilde f(t,s_2
z_2,\ldots,s_n z_n))z_1,z_2,\ldots,z_n\right),
\]
\[
\theta_{s,a}:(t,z_1,\ldots,z_n)\longmapsto \left(t,\frac{s'_1}{s_1}
a t^\kappa
z_1,\frac{s'_2}{s_2}z_2,\ldots,\frac{s'_n}{s_n}z_n\right).
\]
We have $\partial W\cap\beta_\trop =\{b,b'\}$, say
$b\in\fou,b'\in\fou'$.  Let $S:\fou\ra\Spec
\CC[\Lambda_\fou]$ and $S':\fou'\ra\Spec
\CC[\Lambda_{\fou'}]$ be the respective sections of the moment
map.  We have $(n-1)$-tori  $T_e\cdot S(b)$ and $T_e\cdot S'(b')$
that we transport to the nearby fibres by the respective
trivializations of $f$. We then want to connect them in each fibre
of $f$. We first attach a chain $\Gamma_{b'}$ to $T_e\cdot S(b)$
that we construct in a similar way as we built $\Gamma_p$ at a
proper wall point.  It takes care of the transformation $\theta_f$. 
Let $S(b)=(r_1,\ldots,r_n)$ and $S'(b')=(r'_1,\ldots,r'_n)$ be the
respective coordinates. As a second step we then need to connect
$T_e\cdot(r_1,\ldots,r_n)$ to $T_e\cdot (\frac{s'_1}{s_1}a t^\kappa
r'_1,\frac{s'_2}{s_2}r'_2,\ldots,\frac{s'_n}{s_n}r'_n)$ which we do
by straight line interpolation. Let $\log$ denote the inverse of
$\exp$ with imaginary part in $[0,2\pi)$.  For $\lambda\in[0,1]$ and
$z\in\CC^*$, we set $z^\lambda:=\exp(\lambda\log(z))$. For $t\in
T_e$ we map $[0,1]$ into $U$ by
\[
\lambda\longmapsto \left(\left(\frac{s'_1}{s_1}a
t^\kappa\frac{r_1'}{r_1}\right)^\lambda r_1,
\left(\frac{s'_2}{s_2}\frac{r_2'}{r_2}\right)^\lambda r_2,
\ldots,
\left(\frac{s'_n}{s_n}\frac{r_n'}{r_n}\right)^\lambda r_n\right).
\]
Taking the $T_e$-orbit of the image of $[0,1]$ traces out a chain
$\Gamma_p$ and the concatenation of $\pm\Gamma_p$ and
$\pm\Gamma_{b'}$ closes $\beta$ over the slab point $p$ (the sign
$\pm$ being $+$ if and only if $e$ traverses from $\fou$ to
$\fou'$ at $p$).

%===========================================================
%
%		Computation of the period integrals
%
%===========================================================

\section{Computation of the period integrals}
\label{section-compute-periods}
As before, we assume $B$ to be oriented, so we have a canonical
$n$-form $\Omega\in\Omega_{X_0^\dagger/0^\dagger}$  given by $\dlog
z^{v_1}\wedge\ldots\wedge \dlog z^{v_n}$ for $v_1,\ldots,v_n$ an
oriented basis $\Lambda_\sigma$ for any maximal cell $\sigma$. The
orientation of $B$ ensures $\Omega$ is independent of $\sigma$. Since
$\Omega^n_{\shX^\dagger/\Tbs^\dagger}\cong f^*\shO_\Tbs$, there is a
canonical extension of $\Omega$ to $\shX$ by requiring
\[
\int_\alpha\Omega=\operatorname{const}
\]
as a function on $\Tbs$ for $\alpha\cong (S^1)^n$ a vanishing
$n$-cycle at a deepest point of $X_0$ (any two choices for $\alpha$
are homologous). 
\begin{lemma} 
\label{lemma-intalpha}
$\displaystyle\int_\alpha\Omega=(2\pi i)^n.$
\end{lemma}
\begin{proof}
This can be checked in any local chart of a deepest point,  so let
$v\in\P$ be a vertex and $P_v$ the associated monoid such that
$\shX$ in a neighbourhood of the deepest stratum $x\in X_0$
corresponding to $v$ is given by $U=\Spec\CC[P_v]$ and $f=z^\rho$
with $\rho\in P_v$.  The fibre $f^{-1}(1)$ is given by 
\[
\{\phi\in\Hom(P_v,\CC)\mid \varphi(\rho)=1\}
\]
which contains the vanishing $n$-cycle as 
\[
\alpha=\{\phi\in\Hom(P_v,S^1)\mid \varphi(\rho)=1\}=\Hom(P_v/\rho,S^1)
\]
Note that $P_v/\rho = \Lambda_v$, so $\alpha = \Hom(\Lambda_v,S^1)$
which we can homotope to the central fibre and integrate there. Let
$v_1,\ldots,v_n$ be an oriented basis of $\Lambda_v$ then
$\Omega=\dlog z^{v_1}\wedge\ldots\wedge \dlog z^{v_n}$. Set
$z^{v_j}=r_je^{i\theta_j}$, so $\dlog z^{v_j}= \dlog r_j +
id\theta_j$.  Since $\alpha$ is parametrized by having each
$\theta_j\in[0,2\pi)$, we obtain
\[
\int_\alpha\Omega = \int_{\theta_1,\ldots,\theta_n}
i^nd\theta_1\wedge\ldots\wedge d\theta_n = (2\pi i)^n.
\]
\end{proof}

We will make use of the following simple lemma.
\begin{lemma} 
\label{common-lemma}
Let $0\neq \xi\in \ZZ^n$ and let
\[
T=\{\phi\in\Hom(\ZZ^n,\RR/\ZZ)\mid \phi(\xi)=0\}
\]
be the possibly disconnected codimension one subgroup of
$\Hom(\ZZ^n,\RR/\ZZ)$. Let $\theta_1,\ldots,\theta_n$ be a basis of
$\ZZ^n$ and therefore a set of functions $\phi\mapsto
\phi(\theta_j)$ on $T$.  Their duals $\theta_1^*,\ldots,\theta_n^*$
are standard coordinates on $\RR^n$. To give $T$ an orientation, it
suffices to give it for its identity component.  We define a basis
$v_2,\ldots,v_n$ of $T_0 T$ to be oriented if $\xi,v_2,\ldots,v_n$
is an oriented basis of $\ZZ^n$ for its standard orientation. We
have
\[
\int_T d\theta_2\ldots d\theta_n = \langle \theta_1^*,\xi\rangle
\] 
and more generally
\[
\int_T d\theta_1\ldots\widehat{d\theta_j}\ldots d\theta_n =
(-1)^{j+1}\langle \theta_j^*,\xi\rangle.
\] 
\end{lemma}
\begin{proof}
The more general statement follows from the first statement by
permutation and relabelling of variables. The map from $T$ to
$\operatorname{span}\{\theta_2^*,\ldots,\theta_n^*\}/\ZZ^{n-1}$ has
degree $\langle \theta_1^*,\xi\rangle$.  Indeed, this map is
obtained by applying $\Hom(\cdot,\RR/\ZZ)$ to $\oplus_{i=2}^n\ZZ
e_i\ra\ZZ^n/\xi\ZZ$ and this map has determinant $\langle
\theta_1^*,\xi\rangle$.
\end{proof}
The purpose of this section is to prove
Theorem~\ref{period-computed}. When we write $\int_\beta\Omega$ in
the following, we understand this integral as a function in $t$ for
$t\neq 0$.  In particular, we implicitly transport $\beta$ into all
nearby fibres.  The ambiguity of doing so due to monodromy will be
reflected by the circumstance that the logarithm in the base
coordinate appears. Without further mention, we thus implicitly
choose a branch cut in the base.  The choice implicitly made here is
undone upon exponentiating as we eventually do to obtain $h_\beta$
for the assertion of Theorem~\ref{period-computed}.

In the coming sections we first deal with the integral of $\Omega$
over each piece of $\beta$ individually before we finally put the
results together. In the computation we pretend that the patching of
our $k$-th order approximation agrees with the patching of $\shX$.
The result of the period computation can, however, be expressed
purely in terms of the result of applying the logarithmic
deriviative $t\frac{d}{dt}$ to the period integrals and hence, to
order $k$, only depends on the behaviour of the patching to the same
order. Thus this presentational simplification is justified.

%===========================================================

\subsection{Integration over the wall-add-in $\Gamma_p$}
\label{section-wall-integral}
We assume the setup and notation of \S\ref{sec-beta-wall}.  We have
a parametrization of $\Gamma_p$ in the fibre of $f$ for fixed $t$
given by
\[
\phi:(T_e\cdot S(p))\times [0,1]\to U,\qquad
(z_1,\ldots,z_n,\lambda)\longmapsto (1+\lambda \tilde
f(t,z_2,\ldots,z_n))z_1,z_2,\ldots,z_n)
\]
Pulling back $\Omega$ under $\phi$ yields
\[
\begin{array}{rcl}
\phi^*\Omega&=&\Omega+\dlog(1+\lambda \tilde
f(t,z_2,\ldots,z_n))\wedge\dlog z_2\wedge\ldots\wedge\dlog z_n\\
&=&\Omega+\partial_\lambda \log(1+\lambda \tilde
f(t,z_2,\ldots,z_n))d\lambda\wedge\dlog z_2\wedge\ldots\wedge\dlog
z_n.\\
\end{array}
\]
As the first summand $\Omega$ is constant in $\lambda$, it is
trivial as an $n$-form on $(T_e\cdot S(p))\times [0,1]$ and we only
need to deal with the second term. Let $u_1,\ldots,u_{n-1}$ be
standard coordinates on $\RR^{n-1}$ that we use as coordinates on
$\RR^{n-1}/\ZZ^{n-1}$.  Set $z_j=r_je^{i\theta_j}$ and let 
\[
\theta_j=\sum_{k=1}^{n-1} a_{k,j} u_k
\] 
for $1\le j\le n$ be a linear parametrization of $T_e$ where $u_j\in[0,1)$. 
If $S(p)=(r_1,\ldots,r_n)$ then we have
\begin{align*}
\int_{\Gamma_p}\Omega& =\int_{(T_e \cdot S(p))
\times[0,1]}\phi^*\Omega\\
&=\int_{u_1,\ldots,u_{n-1}}\int_\lambda \partial_\lambda
\log(1+\lambda \tilde
f(t,r_2e^{i\theta_2},\ldots,r_ne^{i\theta_n}))d\lambda\dlog
(r_2e^{i\theta_2})\ldots\dlog (r_ne^{i\theta_n})\\
&= c \int_{u_1,\ldots,u_{n-1}}\int_\lambda \partial_\lambda
\log(1+\lambda \tilde f(t,r_2e^{i\sum a_{k,2}u_k},\ldots,r_ne^{i\sum
a_{k,n}u_k}))d\lambda d u_1 \ldots d u_{n-1}
\end{align*}
where $c\in\CC$ equals $i^{n-1}$ multiplied by a polynomial
expression in the $a_{k,j}$. Integrating out $\lambda$ leads to
\[
\int_{\Gamma_p}\Omega= c\int_{u_1,\ldots,u_{n-1}} \log(1+\tilde
f(t,r_2e^{i\sum a_{k,2}u_k},\ldots,r_ne^{i\sum a_{k,n}u_k})) d u_1
\ldots d u_{n-1}.
\]
This integral vanishes up to $t$-order $k$ by the following lemma
combined with the normalization condition \eqref{ncond}.

\begin{lemma} 
Let $m\in\CC[z_1,\ldots,z_{n-1}]$ be a monomial then
\[
\int_{u_1,\ldots,u_{n-1}} m(e^{iu_1},\ldots,e^{iu_{n-1}})du_1\ldots
du_{n-1}\neq 0
\]
if and only if $m$ is constant.
\end{lemma}

%===========================================================

\subsection{Integration over the vertex-add-in $\Gamma_v$}
\label{section-vertex-integral}
We assume the setup and notation of \S\ref{sec-beta-vertex}. We have
$\Gamma_v\subset \mu^{-1}(v)=\Hom(\Lambda_v,S^1)$. 

\begin{lemma} 
\label{lemmaav}
Let $v\in\beta_\trop$ be a vertex of valency $V$. Then for an oriented
basis $v_1,\ldots,v_n$ of $\Lambda_v$ and
$z^{v_j}=r_je^{i\theta_j}$, we get
\[
\int_{\Gamma_v}\Omega=i^n \int_{\Gamma_v}
d\theta_1\ldots d\theta_n \in \left\{ \begin{array}{ll} (2\pi i)^n \ZZ & V\hbox{ is even,}\\ (2\pi i)^n \frac12\ZZ & V\hbox{ is odd.}\end{array} \right.
\]
\end{lemma}
\begin{proof} 
We have $\sum_{v\in e} \eps_{e,v} \xi_e=0$. Set $\xi_j:=\eps_{e_j,v} \xi_{e_j}$ for $e_1,...,e_r$ an enumeration of the edges containing $v$.
\begin{figure}
\resizebox{0.6\textwidth}{!}{
\includegraphics{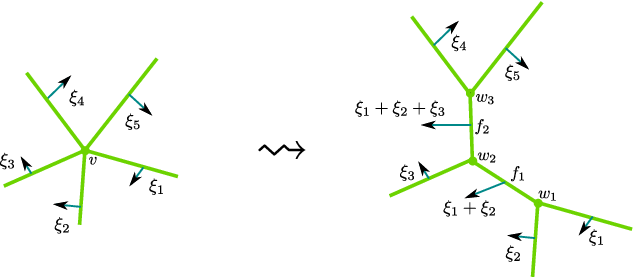}
}
\caption{Making a vertex trivalent by the insertion of new edges}
\label{Fig: make-trivalent}
\end{figure}
We decompose $v$ into trivalent vertices via insertion of $V-3$ new edges $f_1,...,f_{V-3}$ meeting the existing edges in the configuration depicted in Figure~\ref{Fig: make-trivalent}.
Precisely, we replace $v$ by a chain of new edges $f_1,...,f_{V-3}$ such that the ending point of $f_j$ is the starting points of $f_{j+1}$. Let $w_1,...,w_{V-2}$ denote the vertices in this chain.
We arrange it so that $w_1$ meets $e_1,e_2$, $w_2$ meets $e_3$, $w_3$ meets $e_4$ and so forth, finally $w_{V-2}$ meets $e_{V-1},e_V$. 
The edge $f_j$ is decorated with the section $\xi_1+...+\xi_{j+1}$. One checks that at each vertex $w_j$ the balancing condition holds.
One also checks that the new tropical curve is homologous to the original one. Adding boundaries of suitable 2-cycles, we can successively slide down the edges $e_3$, $e_4$,... to $w_1$. In this process the sections along $f_1,...,f_{V-3}$ get modified and when all $e_j$ have been moved to the first vertex, the sections of the $f_j$ are all trivial and so we end up in the original setup setting $w_1=v$. Similarly, one checks that the associated $n$-cycles to the original and modified $\beta_\trop$ are seen to be homologous. Hence 
$$\int_{\Gamma_v}\Omega =  \int_{\Gamma_{w_1}}\Omega +\ldots + \int_{\Gamma_{w_{V-2}}}\Omega.$$
It is not hard to see that we have reduced the assertion to the case where $v$ is trivalent. So we assume $V=3$ now.
As before, set $\xi_j:=\eps_{e_j,v} \xi_{e_j}$ for $j=1,2,3$. 
By the balancing condition, the saturated integral span $V$ of $\xi_1,\xi_2,\xi_3$ has either rank one or two.
In either case, we have a product situation where we can split $\Lambda_v\cong V\oplus W$ which yields a split of the torus
$$\Hom(\Lambda_v,S^1)\cong (V\otimes_\ZZ\RR)/V \times (W\otimes_\ZZ\RR)/W$$
and $\Gamma_v$ also splits as $\bar\Gamma_v\times (W\otimes_\ZZ\RR)/W$. Integration over $(W\otimes_\ZZ\RR)/W$ yields a factor of $(2\pi i)^{n-1}$ when $V$ is one-dimensional and $(2\pi i)^{n-2}$ when $V$ is two-dimensional.
We may thus assume that $\Lambda_v=V$.

We do the one-dimensional case first. 
Let $e$ be a primitive generator of $V$ and $\xi_j=a_j e$. We have $a_1+a_2+a_3=0$. 
Now, $\Gamma_v$ is a union of intervals that connect the points given by $\frac1{a_1}\ZZ$, $\frac1{a_2}\ZZ$, $\frac1{a_3}\ZZ$ on $\RR/\ZZ$. 
We want to show that the signed area of $\Gamma_v$ is $\frac12(2\pi i)$.
We may assume that $a_1,a_2,a_3$ are pairwise coprime, since the non-coprime case is a finite cover of the coprime case and the relative area of $\Gamma_v$ over the entire circle doesn't change when going to the cover.
So let $a_1,a_2,a_3$ be pairwise coprime. The points on the circle are then all distinct except for the origin. 
At the origin, we have three points with both signs appearing, so after sign cancellation, we have distinct points everywhere. 
The signs of the points alternate. This implies that $\Gamma_v$ consists of every other interval between the points. Furthermore these interval all carry the same sign.
The set of points is symmetric under the involution $\iota: x\mapsto -x$. 
However, $\iota$ takes $\Gamma_v$ to its complement, so up to sign $\Gamma_v$ and $\iota(\Gamma_v)$ have the same area. 
Since
$$2\pi i=\int_{\Gamma_v}d\theta + \int_{-\iota(\Gamma_v)}d\theta,$$
we conclude that $\int_{\Gamma_v}d\theta = \frac12(2\pi i)$.
The case when $V$ is two-dimensional works similar.
\end{proof}

%===========================================================

\subsection{Integration over the slab-add-in $\Gamma_v$}
\label{section-slab-integral}
We assume the setup and notation of \S\ref{sec-beta-slab}.  We
filled in the two cylinders $\Gamma_p$ and $\Gamma_{b'}$ at the
slab. We have $\int_{\Gamma_{b'}}\Omega =0$ by
\S\ref{section-wall-integral}.  We parametrize $\Gamma_p$ by
$\phi:T_e\times [0,1]\to U$ given as
\[
(e^{i\theta_1},\ldots,e^{i\theta_n},\lambda) \longmapsto
\left(\left(\frac{s'_1}{s_1}a
t^\kappa\frac{r_1'}{r_1}\right)^\lambda
r_1e^{i\theta_1},
\left(\frac{s'_2}{s_2}\frac{r_2'}{r_2}\right)^\lambda
r_2 e^{i\theta_2},
\ldots,
\left(\frac{s'_n}{s_n}\frac{r_n'}{r_n}\right)^\lambda
r_n e^{i\theta_n}\right).
\]
Note that 
\[
\dlog\left(\left(\frac{s'_j}{s_j}\frac{r_j'}{r_j}\right)^\lambda r_j
e^{i\theta_j}\right)=
\log\left(\frac{s'_j}{s_j}\frac{r_j'}{r_j}\right)d\lambda +
id\theta_j.
\]
We compute
\[
\begin{array}{rcl}
\phi^*\Omega &=& \displaystyle\log\left(\frac{s'_1}{s_1}a
t^\kappa\frac{r_1'}{r_1}\right) i^{n-1} d\lambda \wedge
d\theta_2\wedge\ldots\wedge d\theta_n +\ldots\\
&& + \displaystyle\log\left(\frac{s'_n}{s_n}\frac{r_n'}{r_n}\right)
i^{n-1}  d\theta_1\wedge\ldots\wedge d\theta_{n-1}\wedge
d\lambda.\\[4mm]
&=& \displaystyle\log\left(a t^\kappa\right) i^{n-1} d\lambda \wedge
d\theta_2\wedge\ldots\wedge d\theta_n \\
&&+ \displaystyle\sum_{j=1}^n
(-1)^{j+1}\log\left(\frac{s'_j}{s_j}\frac{r_j'}{r_j}\right) i^{n-1}
d\lambda\wedge
d\theta_1\wedge\ldots\wedge\widehat{d\theta_j}\wedge\ldots\wedge
d\theta_{n-1}.\\
\end{array}
\]
Integrating out $\lambda$ and Lemma~\ref{common-lemma} yield
\begin{equation}
\label{slab-integral}
\begin{array}{rcl}
\displaystyle\int_{\Gamma_p}\Omega=\int_{T_e} \int_{\lambda}
\phi^*\Omega
&=& \displaystyle (2\pi i)^{n-1}\log\left(a t^\kappa\right) \langle
v_1^*,\xi_e \rangle \\
&&+(2\pi i)^{n-1} \displaystyle\sum_{j=1}^n
\log\left(\frac{s'_j}{s_j}\frac{r_j'}{r_j}\right) \langle
v_j^*,\xi_e\rangle.
\end{array}
\end{equation}
We record for later use that evaluating the logarithms, we find
among the summands
\begin{equation}
\label{record-whats-to-be-cancelled}
(2\pi i)^{n-1} \displaystyle\sum_{j=1}^n \log (r_j')\langle
v_j^*,\xi_e \rangle \,-\, (2\pi i)^{n-1} \displaystyle\sum_{j=1}^n
\log (r_j)\langle v_j^*,\xi_e \rangle.
\end{equation}
Furthermore using \eqref{s-in-terms-of-sj}, we find
\[
\sum_{j=1}^n \log\left(\frac{s'_j}{s_j}\right) \langle
v_j^*,\xi_e\rangle = \log
\left(\frac{s_{\fob,\sigma_{\fou'}}(\xi)}{
s_{\fob,\sigma_{\fou}}(\xi)}\right). 
\]
As $\fob,\fou, \frak{u'}$ and $a$ are determined by $p$, we
may shortcut
\begin{equation}
\label{def-sp}
s_p:=a
\frac{s_{\rho_\fob,\sigma_{\fou'}}(\xi)}{s_{\rho_\fob,\sigma_{\fou}}
(\xi)}:\Lambda_p\lra \CC^{*}
\end{equation}
if $e$ traverses from $\fou$ to $\fou'$ at $p$ and we define
$s_p$ to be the inverse of \eqref{def-sp} if $e$ traverses in the
opposite direction.

%===========================================================

\subsection{Integration over part of an edge}
\label{section-chamber-integral}
Let $c$ be a connected part of an edge $e$ of $\beta$ that lies in
the interior of a maximal cell $\sigma=\sigma_{\fou}$ of $\P$ and is
contained in the closure of a chamber $\fou$. As in
\S\ref{section-slab-integral}, we parametrize $T_e\cdot S(c)$ by
$\phi: T_e \times [0, 1] \ra\Spec \CC[\Lambda_\sigma]$ so that
$\phi(x,[0,1])$ is co-oriented with $e$. In the coordinates
$z^{v_1},\ldots,z^{v_n}$ let $(r_1,\ldots,r_n)$ and
$(r_1',\ldots,r_n')$ be the starting and ending points,
respectively, of the lift of $c$ under the section of the moment
map. From the calculation in \S\ref{section-slab-integral}, by
setting $s_i=s_i'=a=1$ and $\kappa=0$, we obtain
\[
\int_{T_e\cdot S(c)}\Omega =  (2\pi i)^{n-1}
\displaystyle\sum_{j=1}^n \log (r_j')\langle v_j^*,\xi_e \rangle
\,-\, (2\pi i)^{n-1} \displaystyle\sum_{j=1}^n \log (r_j)\langle
v_j^*,\xi_e \rangle.
\]

%===========================================================

\subsection{Integration in a neighbourhood of a vertex}
Let $W\in\shW$ be a neighbourhood containing a vertex $v$ of $\beta$
and let $U\in\shU$ be the corresponding open set in $\shX$. 
The cycle $\beta_U:=\beta\cap U$ consists of the vertex-filling
cycle $\Gamma_v$ and various cylinders $\Gamma_{c_k}$, $k=1,\ldots,r$
one for each edge $e_k$ meeting $v$. We may thus compute
$\int_{\beta_U}\Omega$ by adding the results from
\S\ref{section-vertex-integral} and
\S\ref{section-chamber-integral}. Say the moment map lift of $v$ has
coordinates $(r_1,\ldots,r_n)$ and the other endpoint of the lift of
$c_k$ is $(r'_1(k),\ldots,r'_n(k))$ then
\begin{eqnarray*}
\int_{\beta_U}\Omega &= &\int_{\Gamma_v}\Omega
+\sum_{k=1}^r \int_{\Gamma_{c_k}}\Omega \\
&=& (2\pi i)^n a_v +(2\pi i)^{n-1} \sum_{k=1}^r \eps_{e_k,v}
\Big(\sum_{j=1}^n  \log (r'_j(k))\langle v_j^*,\xi_e \rangle
-\sum_{j=1}^n  \log (r_j)\langle v_j^*,\xi_e \rangle\Big).
\end{eqnarray*}
Using the balancing condition \eqref{balancing}, this reduces to
\[
\int_{\beta_U}\Omega = (2\pi i)^n a_v + (2\pi i)^{n-1} \sum_{k=1}^r
\eps_{e_k,v} \sum_{j=1}^n \log (r'_j(k))\langle v_j^*,\xi_e \rangle.
\]

%===========================================================

\subsection{Integration in the neighbourhood of a wall}
Let $W\in\shW$ be a neighbourhood containing a wall point $p$ of
$\beta$ and let $U\in\shU$ be the corresponding neighbourhood of
$\shX$.  The cycle $\beta_U:=\beta\cap U$ consists of the
wall-point-filling cycle $\Gamma_v$ and two cylinders $\Gamma_{c_1}$,
$\Gamma_{c_2}$ for $c_1,c_2\subset e$ for $e$ the edge of $\beta$
that contains $p$. Say $c_1,c_2$ have the same orientation as $e$
and are ordered by the orientation as well. Since the integral over
$\Gamma_p$ is zero by \S\ref{section-wall-integral} and since the
endpoint of $c_1$ coincides with the starting point of $c_2$, using
\S\ref{section-chamber-integral} and cancellation at $p$, we obtain
\[
\int_{\beta_U}\Omega = (2\pi i)^{n-1} \displaystyle\sum_{j=1}^n \log
(r_j')\langle v_j^*,\xi_e \rangle \,-\, (2\pi i)^{n-1}
\displaystyle\sum_{j=1}^n \log (r_j)\langle v_j^*,\xi_e \rangle.
\]
with $(r_1,\ldots,r_n)$ the starting point of $c_1$ and
$(r'_1,\ldots,r'_n)$ the endpoint of $c_2$.

%===========================================================

\subsection{Proof of Theorem~\ref{period-computed}}
We have 
\[
\int_\beta\Omega =\sum_{U\in\shU}\int_{\beta\cap U} \Omega.
\]
We computed the summands in the previous section and just need to
add the results. We show that all terms of the form
\[
(2\pi i)^{n-1} \displaystyle\sum_{j=1}^n \log (r_j)\langle
v_j^*,\xi_e \rangle
\]
with varying $(r_1,\ldots,r_n)$ cancel.  We claim that whenever to
cylinders $\Gamma_{c_1}$, $\Gamma_{c_2}$ share an endpoint, the
corresponding terms for these endpoints carry opposite sign and thus
cancel.  This is in fact easy to see if $c_1,c_2$ lie in the same
maximal cell $\sigma$ as the coordinates $z^{v_1},\ldots,z^{v_n}$
that we used to write down $(r_1,\ldots,r_n)$ with $v_j$ a basis of
$\Lambda_\sigma$ extend to all open sets of the form $U\cap X_0$
with $U$ meeting no other component of $X_0$ than $X_\sigma$. It
remains to take a closer look at a slab $\fob$ where a
transition from $\Lambda_{\sigma}$ to $\Lambda_{\sigma'}$ occurs
($\sigma\neq\sigma'$).  Here, we can identify $\Lambda_{\sigma'}$
with ($\sigma\neq\sigma'$) via parallel transport through
$\fob$. Then \eqref{record-whats-to-be-cancelled} implies the
desired cancellation there as well.

Now Theorem~\ref{period-computed} directly follows up to signs. To
check the signs note that $d_p$ in the statement of the theorem
points in the opposite direction as $v_1^*$ in the definition of the
slab add-in.  This induces a negative sign that cancels with the
negative sign in the multiplication with $-2\pi i$ in the definition
of $h_\beta$. Hence, we are done.

%===========================================================
%
%		Tropical cycles generate all cycles
%
%===========================================================

\section{Tropical cycles generate all cycles}
\label{sec-W2W0}
\begin{theorem} 
\label{main-affine-result}
Let $(B,\P)$ be an oriented tropical manifold. We have
\begin{align}
H^{n-1}(B,i_*\Lambda)&= H_{1}(B,\partial B;i_*\Lambda)
\label{main-affine-result-1}\\
H_{1}(B,\partial B;i_*\Lambda) &= H_{1}^{\P^{\bary}}(B,\partial
B;i_*\Lambda) \label{main-affine-result-2}
\end{align}
\end{theorem}
\begin{proof}
\eqref{main-affine-result-1} is Theorem~\ref{cor-iso-ho-coho} and
\eqref{main-affine-result-2} is Theorem~\ref{simplicial=singular}.
\end{proof}

\begin{theorem} 
\label{shape-of-cycles}
Let $(B,\P)$ be a tropical manifold.  Any element in
$H_{1}(B,\partial B;i_*\Lambda)$ can be represented by a tropical
$1$-cycle.
\end{theorem} 
\begin{proof} 
By Theorem~\ref{main-affine-result}, (2) we are dealing with an
element of $H_{1}^{\P^{\bary}}(B,\partial B;i_*\Lambda)$ which is a
similar object as a tropical $1$-cycle. However, its vertices are
the barycenters of $B$.  We can ignore edges in $\partial B$.  We
perturb the vertices by adding the boundary of a $2$-chain of
$H_{1}(B,\partial B;i_*\Lambda)$ so that the vertices that were in
$\partial B$ stay in $\partial B$ and become univalent and the
vertices in the interior of $B$ all move into the maximal cells of
$B$ so that the new edges between them do not meet $\Delta$.
\end{proof}

For a treatment of to log-differential forms in our setup, see
\cite{logmirror2}.  In short, let $\shZ$ denote the locus in $\shX$
where $f$ is log-singular and $j:\shX\setminus \shZ\hra \shX$ the
inclusion of its complement. One also denotes by $j:X_0\setminus
(\shZ\cap X_0)\hra X_0$ the analogous version on $X_0$.

\begin{theorem}[Base change for $H^{1,n-1}$] 
Assume that $B$ is simple. We have that
$H^{n-1}(\shX, j_*\Omega_{\shX^\dagger/\Tbs^\dagger})$ is a free
$\shO_\Tbs$-module and its formation commutes with base change, so
in particular
\[
H^{n-1}(X_0,j_*\Omega_{X_0^\dagger/0^\dagger})\cong
H^{n-1}(X_t,\tilde\Omega_{X_t})
\]
for $t\neq 0$ where $\tilde\Omega_{X_t}$ are the
Danilov-differentials, i.e. the pushforward of the usual
differentials from the non-singular locus of $X_t$.  Furthermore,
there is a canonical isomorphism
\[
H^{n-1}(X_0,j_*\Omega_{X_0^\dagger/0^\dagger})=
H^{n-1}(B,i_*\Lambda\otimes_\ZZ\CC).
\]
\end{theorem} 

\begin{proof}
We first prove the second assertion.  Note that in the references we
are going to cite, $B$ denotes the dual intersection complex instead
of the intersection complex, hence our $\Lambda$ will be
$\check\Lambda$ in the references. We already make this adaption
upon citing. By \cite[Thm 1.11]{Ru10} there is an injection
\begin{equation}
\label{injection-affine-to-log}
H^{n-1}(B,i_*\Lambda\otimes_\ZZ\CC)\hra
H^{n-1}(X_0,j_*\Omega_{X^\dagger_0/0^\dagger})
\end{equation}
which is an isomorphism if $B$ is simple by \cite[Theorem
3.22]{logmirror2}. For the first assertion, by \cite[Theorem
4.1]{logmirror2}, 
\[
H:=\HH^{n}(\shX,j_*\Omega^\bullet_{\shX^\dagger/\Tbs^\dagger})
\] 
satisfies base change. By \cite[Thm 1.1 b)]{Ru10},
$H^{n-1}(B,i_*\Lambda\otimes_\ZZ\CC)$ is a sub-quotient of $H$, a
graded piece of the stupid filtration. This implies the statement.
\end{proof}

Let $W_\bullet$ denote the monodromy weight filtration of $f$ on
$H_n(X_t,\QQ)$ for $t\neq 0$, see \cite[2.4]{De93}. The vanishing
$n$-cycle $\alpha$ generates $W_0$. Analogously one obtains a
filtration on $H^n(X_t,\QQ)$ transforming into the previous one
under Poincar\'e duality.
\begin{theorem} If $B$ is simple then
\label{thm-identifyW2W1}
$W_2/W_1 = H^{n-1}(B,i_*\Lambda\otimes_\ZZ\CC)$.
\end{theorem} 
\begin{proof}
If $T$ denotes the monodromy operator, we have that $N=\log T$ acts
on the subspace $\bigoplus_{p+q=n}
H^{q}(B,i_*\bigwedge^p\Lambda\otimes_\ZZ\CC)$ of $H$ as
\[
N:H^{q}(B,i_*\bigwedge^p\Lambda\otimes_\ZZ\CC)\lra
H^{q+1}(B,i_*\bigwedge^{p-1}\Lambda\otimes_\ZZ\CC)
\]
by cupping with the radiance obstruction class in
$H^1(B,i_*\check\Lambda)$, see \cite[Theorem~5.1]{logmirror2}. We
need to show that 
\begin{equation}
\label{N-power}
N^{n-2}:H^{1}(B,i_*\bigwedge^{n-1}\Lambda\otimes_\ZZ\CC)\lra
H^{n-1}(B,i_*\Lambda\otimes_\ZZ\CC)
\end{equation}
is an isomorphism. This follows from the mirror symmetry result
proven in \cite[Theorem~5.1]{logmirror2}: $N$ is the Lefschetz
operator on the mirror dual of $X_0$ for which \eqref{N-power} is
known to be an isomorphism by the Lefschetz decomposition theorem.
\end{proof}

\begin{corollary}[Generation of $W_2/W_0$] 
\label{genW2W0}
If $B$ is simple then tropical $1$-cycles generate $W_2/W_0$.
\end{corollary}

\begin{proof}
We have $W_1=W_0$. Combine Theorem~\ref{thm-identifyW2W1} with
Theorem~\ref{enough-cycles}.
\end{proof}

%===========================================================
%
%		Appendix: cohomology and homology of constructible sheaves
%
%===========================================================

\section{Appendix: cohomology and homology of constructible sheaves}

%===========================================================

\subsection{Identification of simplicial and singular homology with
coefficients in a constructible sheaf}
\label{general-constructible-sheaves}
We could not find the following results on constructible sheaves in
the literature, so we provide proofs here. Recall from
\cite[\S2.1]{Ha02} that a \emph{$\triang$-complex} is a CW-complex
where each closed cell comes with a distinguished surjection to it
from the (oriented) standard simplex with compatibility between
sub-cells and faces of the simplex.  For a $\triang$-complex $X$, we
use the notation $X=\coprod_{\tau\in T} \tau^\circ$ where $T$ is a
set of simplices for each of which we have the \emph{characteristic
map} $j_\tau:\tau\ra X$  that restricts to a homeomorphism on the
interior $\tau^\circ$ of $\tau$. Let $X=\coprod_{\tau\in
T}\tau^\circ$ be a $\triang$-complex. We say a sheaf $\shF$ on $X$
is \emph{$T$-constructible} if $\shF|_{\tau^\circ}$ is a constant
sheaf for each $\tau\in T$. Let $A=\coprod_{\tau\in S}\tau^\circ$ be
a (closed) subcomplex ($S\subseteq T$). 
\begin{definition}
\label{def-homology}
\begin{enumerate}
\item
We denote by $H_i^T(X,A;\shF)$ the relative simplicial homology with
coefficients $\shF$, i.e. it is computed by the differential graded
vector space 
\[
\bigoplus_{i\ge 0} C^T_i(X,A;\shF)\qquad \hbox{ where }\qquad
C^T_i(X,A;\shF)=\bigoplus_{{\tau\in T\setminus
S}\atop{\dim\tau=i}}\Gamma(\tau,j_\tau^*\shF)
\]
with the usual differential $\partial:C^T_i(X,A;\shF)\ra
C^T_{i-1}(X,A;\shF)$ whose restriction/projection to
$\Gamma(\sigma,j_\sigma^*\shF)\ra \Gamma(\tau,j_\sigma^*\shF)$ for
$\tau\subset\sigma$ a facet inclusion is given by the restriction
map multiplied by
\[
\eps_{\tau\subset\sigma}
=\left\{\begin{array}{ll}
+1, & n_\tau\wedge\ori_{\tau}=+\ori_\sigma \\
-1, & n_\tau\wedge\ori_{\tau}=-\ori_\sigma
\end{array}\right.
\]
where $\ori_{\tau},\ori_{\sigma}$ denote the orientation of
$\tau,\sigma$ respectively and $n_\tau$ is the outward normal of
$\sigma$ along $\tau$ (for $\tau$ a point, set
$n_\tau\wedge\ori_{\tau}:=n_\tau$).

\item
On the other hand, one defines $H_i(X,A;\shF)$, the singular
homology with with coefficients $\shF$, in the usual way (see for
example \cite[VI-12]{Br97}) where chains $C_i(X,A;\shF)$ are formal
sums over singular $i$-simplices in $X$ modulo singular
$i$-simplices in $A$.
\end{enumerate}
We denote by $\overrightarrow{C}_i(X,A;\shF)$ the direct limit of
$C_i(X,A;\shF)$ under the barycentric subdivision operator on
singular chains, see \cite[V-1.3]{Br97}. We write $C^T_i$ for
$C^T_i(X,A;\shF)$ and $C_i$ for $C_i(X,A;\shF)$ when the spaces and
the sheaf are unambiguous. We also write $\overrightarrow{C}_i$ for
$\overrightarrow{C}_i(X,A;\shF)$.
\end{definition}

\begin{lemma} 
\label{reduce-to-barycentric}
Let $X=\coprod_{\tau\in T} \tau^\circ$ be a $\triang$-complex, $A$ a
subcomplex and $\shF$ be a $T$-constructible sheaf on $X$. We denote
by $T^\bary$ the barycentric subdivision of $T$. Note that $\shF$ is
$T^\bary$-constructible. For any $i$, there is a natural isomorphism
\[
H^{T}_i(X,A;\shF) \lra H^{T^\bary}_i(X,A;\shF).
\]
\end{lemma}

\begin{proof}
For $\tau\in T^\bary$, let $\hat\tau$ denote the smallest simplex in
$T$ containing $\tau$. The chain complex $C_\bullet^{T^\bary}$
receives a second grading by setting $C_{i,j}^{T^\bary}=
\bigoplus_{{\tau\in T^\bary\setminus S^\bary}\atop{\dim\tau=i},
\dim\hat\tau-\dim\tau=j} \Gamma(\tau,j_\tau^*\shF)$ and the
differential splits $\partial= \partial_1+\partial_2$ into
components corresponding to the indices. Since $\partial_2$ computes
the homology of each cell in $T^\bary\setminus S^\bary$, we have
$C_i^T=H^{\partial_2}_0 (C_{i,\bullet}^{T^\bary})$ and
$H^{\partial_2}_k(C_{i,\bullet}^{T^\bary})=0$ for $k>0$ so that the
spectral sequence $E_1^{p,q}=H^{\partial_2}_{-q}
(C_{-p,\bullet}^{T^\bary})\Rightarrow H^{T^\bary}_{-p-q}(X,A;\shF)$
yields the result.
\end{proof}

Let $j_k$ denote the inclusion of the complement of the
$(k-1)$-skeleton in $X$, i.e. $$j_k:\left(\coprod_{{\tau\in
T}\atop{\dim\tau\ge k}}\tau^\circ \right)\hra X.$$ Consider the
decreasing filtration 
\[
\shF=\shF^{-1}\supset \shF^0\supset\shF^1\supset\ldots
\]
of $\shF$ defined by $\shF^k={(j_k)_!}{j_k^*}\shF$. Let $i_k$ denote
the inclusion of the $k$-skeleton in $X$ and $i_{\tau^\circ}$ denote
the inclusion of $\tau^\circ$ in the $k$-skeleton. We have
\[
\Gr^k_\shF = \shF^k/\shF^{k+1}=\bigoplus_{{\tau\in
T}\atop{\dim\tau=k}} (i_k)_*(i_{\tau^\circ})_!(\shF|_{\tau^\circ}).
\]

\begin{lemma} 
\label{reducetoGr}
We have the following commutative diagram with exact rows
\begin{equation*}
\xymatrix@C=30pt
{
0 \ar[r]& C_\bullet^{T^\bary}(\shF^{k+1})   \ar[r]\ar[d]&
C_\bullet^{T^\bary}(\shF^{k})               \ar[r]\ar[d]&
C_\bullet^{T^\bary}(\Gr^k_\shF)         \ar[r]\ar[d]& 0 \\
0 \ar[r]& \overrightarrow{C}_\bullet(\shF^{k+1}) \ar[r]&      
\overrightarrow{C}_\bullet(\shF^{k}) \ar[r]&      
\overrightarrow{C}_\bullet(\Gr^k_\shF) \ar[r]&       0 \\
}
\end{equation*}
\end{lemma}
\begin{proof}
The vertical maps are the natural inclusions, commutativity is
straightforward. The left-exactness of the global section functor
leaves us with showing the exactness of the rows at the rightmost
non-trivial terms.  For the first row, note that the vertices of any
$\tau\in T^\bary$ are barycenters of simplices in $T$ of different
dimensions, so there is a unique vertex of $\tau$ corresponding to
the lowest-dimensional simplex in $T$.  This allows to apply a
retraction-to-the-stalk argument as in \cite[Proof of
Lemma~5.5]{logmirror1} to show that $H^1(\tau,j_\tau^*\shF)=0$, so
the first row is exact. We show the surjectivity of 
$\overrightarrow{C}_\bullet(\shF^{k}) \ra
\overrightarrow{C}_\bullet(\Gr^k_\shF)$. Let $s:\tau\ra X$ be a
singular simplex and $g\in \Gamma(\tau,s^{-1}(\shF_k/\shF_{k+1}))$.
By the exactness of $s^{-1}$ and the surjectivity of
$\shF_k\ra\shF_k/\shF_{k+1}$, we find an open cover $\{U_\alpha\}$
of $\tau$ such that  $g|_{U_\alpha}$ lifts to
$\hat{g}_\alpha\in\Gamma(U_\alpha,\shF_k)$. By the compactness of
$\tau$, we may assume the cover to be finite. After finitely many
iterated barycentric subdivisions of $\tau$, we may assume each
simplex of the subdivision to be contained in a $U_\alpha$ for some
$\alpha$. Let $\tau'$ be such a simplex contained in $U_\alpha$,
then $g|_{\tau'}$ lifts to $\hat{g}_\alpha|_{\tau'}$ and we are done
since it suffices to show surjectivity after iterated barycentric
subdivision.
\end{proof}

\begin{theorem} 
\label{simplicial=singular}
Let $X=\coprod_{\tau\in T} \tau^\circ$ be a $\triang$-complex, $A$ a
subcomplex and $\shF$ be a $T$-constructible sheaf on $X$.  For any
$i$, the natural map 
\[
H_i^T(X,A;\shF)\lra H_i(X,A;\shF)
\] is an isomorphism.
\end{theorem}

\begin{proof} 
There is also a natural map $H_i^{T^\bary}(X,A;\shF)\ra
H_i(X,A;\shF)$ and by Lemma~\ref{reduce-to-barycentric}, it suffices
to prove that this is an isomorphism. Moreover, by long exact
sequences of homology of a pair, it suffice to prove the absolute
case, so assume $A=\emptyset$.  By the long exact sequences in
homology associated to the rows in the diagram in
Lemma~\ref{reducetoGr} and the five-Lemma, it suffices to prove that
the embedding
\begin{equation}
\label{eq-Cs}
C_\bullet^{T^\bary}(\Gr^k_\shF)\lra \overrightarrow{C}_\bullet
(\Gr^k_\shF)
\end{equation}
induces an isomorphism in homology. The problem is local, so fix
some $k$-simplex $\sigma\in T$ and let $v_\sigma$ denote the
barycenter of $\sigma$. We define the open star of $v_\sigma$, a
contractible open neighbourhood of the interior of $\sigma$, by
\[
U=\coprod_{{\tau\in T^\bary}\atop{v_\sigma\in\tau}} \tau^\circ.
\]
Let $M$ be the stalk of $\shF$ at a point in $\sigma^\circ$. Note
that the right-hand side of \eqref{eq-Cs} can be identified with
$\overrightarrow{C}_\bullet(U;M)/ \overrightarrow{C}_\bullet
(U\setminus\sigma^\circ;M)$  where (by abuse of notation) $M$ also
denotes the constant sheaf with stalk $M$ on $U$, so it computes the
singular homology $H_\bullet(U,U\setminus\sigma^\circ;M)$.  Most
importantly, we have reduced the situation to singular homology with
constant coefficients, so we are allowed to apply standard
techniques like deformation equivalences as follows. The pair
$(U,U\setminus\sigma^\circ)$ retracts to $(V,V\setminus v_\sigma)$
where  $V=\coprod_{{\tau\in T^\bary,v_\sigma\in\tau}
\atop{\tau\cap\partial\sigma=\emptyset}} \tau^\circ$.  By excision,
we transition to the pair $(\overline{V},\overline{V} \setminus
v_\sigma)$ where $\overline{V}$ is the closure of $V$ in $X$.  On
the other hand, $\overline{V}\setminus v_\sigma$ retracts to
$\overline{V}\setminus V$ inside $\overline{V}$. Summarizing, we
obtain isomorphisms
\[
H_\bullet(U,U\setminus\sigma^\circ;{M}) = H_\bullet(V,V\setminus
v_\sigma;{M}) = H_\bullet(\overline{V},\overline{V}\setminus
v_\sigma;{M}) = H_\bullet(\overline{V},\overline{V}\setminus V;{M}).
\]
On the other hand, we identify the left-hand side of \eqref{eq-Cs} as
\[
\bigoplus_{k\ge 0}\bigoplus_{{\tau\in
T^\bary,\dim\tau=k}\atop{v_\sigma\in\tau,
\tau\cap\partial\sigma=\emptyset}} M
\]
which coincides with $C_\bullet^{T^\bary\cap \overline V}(\overline
V; M)/ C_\bullet^{T^\bary\cap  (\overline V\setminus V)}(\overline
V\setminus V; M)$  noting that $\overline V$ and $\overline
V\setminus V$ are $\triang$-sub-complexes of $X$. The result follows
from the known isomorphism of simplicial and singular homology for
constant coefficients
\[
H^{T^\bary\cap \overline
V}_\bullet(\overline{V},\overline{V}\setminus
V;{M})=H_\bullet(\overline{V},\overline{V}\setminus V;{M}),
\]
see for example \cite[Theorem 2.27]{Ha02}.
\end{proof}

%===========================================================

\subsection{A general homology-cohomology isomorphism for
constructible sheaves on topological manifolds}
\label{sec-hocoho-compare}
We fix the setup for the entire section.
\begin{setup} 
\label{setup-acyclic}
\begin{enumerate}
\item
Let $\P$ be a simplicial complex and $\Lambda$ a $\P$-constructible
sheaf on its topological realization $B$. We assume there is no
self-intersection of cells in $B$.
\item
We assume that $B$ is an oriented topological manifold possibly with
non-empty boundary $\partial B$. We set $n=\dim B$.
\item
For $\tau$ an $n$-dimensional simplex, let $U_\tau$ denote a small
open neighbourhood of $\tau$ in $B$. We denote $\P^\max=\{\tau\in\P
| \dim\tau=n\}$. We assume that the open cover $\fou=\{U_\tau|
\tau\in \P^\max\}$ is $\Lambda$-acyclic, i.e.
\[
H^i(U_{\tau_1}\cap\ldots\cap U_{\tau_k} ,\Lambda)=0
\]
for $i>0$ and any subset $\{\tau_1,\ldots,\tau_k\}\subseteq \P^\max$.
\end{enumerate}
\end{setup}
Note that $U_{\tau_1}\cap\ldots\cap U_{\tau_k}$ is a small open
neighbourhood of ${\tau_1}\cap\ldots\cap {\tau_k}$.

\begin{example} 
If $\P'$ is a polyhedral complex that glues to an oriented
topological manifold $B$ and $\Lambda$ is a $\P'$-constructible
sheaf then the barycentric subdivision $\P$ of $\P'$ satisfies the
conditions of Setup~\ref{setup-acyclic}. 
%Indeed (3) is deduced via Lemma~\ref{quotient-is-poset} and
%Lemma~\ref{poset-acyclic} as every barycentric simplex contains a
%unique minimal barycenter with respect to the ordering of $\P'$.
\end{example}

Fixing an orientation of each $\tau\in \P$, we can define the chain
complex $C^\P_\bullet(B,\partial B;\Lambda)$ as in
Definition~\ref{def-homology}, in particular 
\begin{equation}
\label{eq-def-homology}
C^\P_i(B,\partial B;\Lambda)  =
\bigoplus_{{\dim\tau=i}\atop{\tau\not\subset\partial
B}}\Gamma(\tau,\Lambda).
\end{equation}
To keep notation simple and since $H^\P_i(B,\partial B;\Lambda)=
H_i(B,\partial B;\Lambda)$ by Theorem~\ref{simplicial=singular},  we
denote $H^\P_i(B,\partial B;\Lambda)$ by $H_i(B,\partial B;\Lambda)$
and also $C^\P_i(B,\partial B;\Lambda)$ by $C_i(B,\partial
B;\Lambda)$. We denote by $C^\bullet(B,\Lambda)$ the \v{C}ech
complex for $\Lambda$ with respect to $\fou$ and some total
ordering of $\P^\max$. Given
$I=\{\tau_0,\ldots,\tau_i\}\subset\P^\max$, we use the notation
$U_I=U_{\tau_0}\cap \ldots\cap U_{\tau_i}$, so 
\begin{equation}
\label{def-cohomology}
C^i(B,\Lambda) = \bigoplus_{|I|=i+1}\Gamma(U_I,\Lambda).
\end{equation}
The purpose of this section is to define a natural isomorphism
$H^\P_\bullet(B,\partial B;\Lambda)\ra H^{n-\bullet}(B; \Lambda)$.
We call $\P$ co-simplicial if the intersection of any set of $k+1$
many maximal cells is either empty or a $(n-k)$-dimensional simplex.
In the co-simplicial case the index sets of the sums of
\eqref{eq-def-homology} and \eqref{def-cohomology} for $C_i^\P$ and
$C^{n-i}$ respectively are naturally in bijection (ignoring empty
$U_I$) and the map would be straightforwardly defined as an
isomorphism of complexes that is on each term given by
\begin{equation}
\label{co-simplicial-case}
\Gamma(\tau_0\cap\ldots\cap\tau_i,\Lambda)
\stackrel{\sim}{\longrightarrow} \Gamma(U_{\tau_0}\cap\ldots\cap
U_{\tau_i},\Lambda)
\end{equation}
up to some sign convention.  Note that in the co-simplicial case, a
cell in $\partial B$ cannot be written as an intersection of maximal
cells, so we need to take homology relative to the boundary. We are
going to show (see Proposition~\ref{prop-map-ho-coho}) that such a
map can be generalized to a $\P$ that is not co-simplicial by
replacing the right-hand side of \eqref{co-simplicial-case} by a
complex  $C^\bullet_{\tau_0\cap\ldots\cap\tau_i}(B,\Lambda)$.  Note
that each non-empty $U_I$ has a unique cell $\tau$ in $\P$ that is
maximal with the property of being contained in it.  Fixing this
cell $\tau$, gathering all terms in the \v{C}ech complex for open
sets $U_I$ where $\tau$ is this unique maximal cell, yields a
subcomplex $C^\bullet_\tau(B,\Lambda)$. In fact, we have a
decomposition of the group $C^\bullet(B,\Lambda)$ by setting
\[
C^i(B,\Lambda)=\bigoplus_{\tau\in\P} C^i_\tau(B,\Lambda) \qquad
\hbox{ where }  \qquad C^i_\tau(B,\Lambda) = \bigoplus_{\left\{
I\left| {|I|=i+1, \tau\subset U_I}\atop{\sigma\not\subset
U_I\hbox{\tiny whenever }\tau\subsetneq\sigma}\right.\right\}}
\Gamma(U_I,\Lambda).
\]
We consider the decreasing filtration $F^\bullet$ by sub-complexes
of $C^\bullet(B,\Lambda)$ given by 
\[
F^kC^i(B,\Lambda) = \bigoplus_{{\tau\in\P}\atop{\codim\tau\ge k}}
C^i_\tau(B,\Lambda).
\]
We define the associated $k$th graded complex by
\[
\Gr_F^kC^i(B,\Lambda) = F^kC^i(B,\Lambda)/F^{k+1}C^i(B,\Lambda)
\]
which yields a direct sum of complexes
\[
\Gr_F^kC^\bullet(B,\Lambda)=\bigoplus_{{\tau\in\P}\atop{\codim\tau=k}}
C^\bullet_\tau(B,\Lambda)
\]
turning each $C^\bullet_\tau(B,\Lambda)$ into a complex.

\begin{lemma} 
\label{lemma-gradedcoho}
We have 
\[
H^i(\Gr_F^kC^\bullet(B,\Lambda))
=\left\{\begin{array}{ll}
\bigoplus_{{\tau\in\P,\tau\not\subset\partial B}\atop{\codim\tau=k}}
\Gamma(\tau,\Lambda) &\hbox{for }i=k,\\
0&\hbox{otherwise.}
\end{array}\right.
\]
\end{lemma}
\begin{proof} Since
\[
\Gr_F^kC^\bullet(B,\Lambda)=\bigoplus_{{\tau\in\P}\atop{\codim\tau=k}}
C_\tau^\bullet(B,\ZZ)\otimes_\ZZ \Gamma(\tau,\Lambda),
\]
it suffices to show that $H^i(C_\tau^\bullet(B,\ZZ))$ is isomorphic
to $\ZZ$ when $\codim\tau=i$ and trivial otherwise.  The set
$\fou_\tau=\{U_\sigma\in\fou| \tau\subset\sigma \}$ covers
an open ball containing $\tau$.  Let $C^\bullet(\fou_\tau,\ZZ)$
denote the associated \v{C}ech complex. We have a short exact
sequence of complexes
\begin{equation}
\label{Utau-ses}
0\lra  C_\tau^\bullet(B,\ZZ) \lra C^\bullet(\fou_\tau,\ZZ)\lra
\overline{C}_\tau^\bullet(\ZZ) \lra 0
\end{equation}
where 
\[
\overline{C}_\tau^i(\ZZ) =
\bigoplus_{{I=\{\sigma_0,\ldots,\sigma_i\}} \atop{{\{U_{\sigma_0},
\ldots,U_{\sigma_i}\}\subset\fou_\tau}
\atop{\tau\neq\sigma_0\cap\ldots\cap\sigma_i}}} \Gamma(U_I,\ZZ)
\]
is the induced cokernel. Denoting
$K=\bigcup_{\sigma\in\P,\tau\subseteq\sigma}\sigma$, one finds the
sequence \eqref{Utau-ses} is naturally identified with a sequence of
\v{C}ech complexes computing the long exact sequence
\[
\ldots\lra H^i_\tau(K,\ZZ)\lra H^i(K,\ZZ)\lra
H^i(K\setminus\tau,\ZZ) \lra\ldots
\]
We have
\[
K\setminus\tau\hbox{ is }
\left\{
\begin{array}{l} 
\hbox{homotopic to }S^d\hbox{ with }d=\codim\tau-1\hbox{ if
}\tau\not\subset\partial B\\ \hbox{contractible if
}\tau\subset\partial B
\end{array}\right.
\]
Since $K$ is contractible we get that $H^i_\tau(K,\ZZ)=0$ for all
$i$ if $\tau\subset\partial B$, that is, $C^\bullet_\tau(B,\ZZ)$ is
exact in this case. Otherwise, we find $H^i_\tau(K,\ZZ)\cong\ZZ$ for
$i=\codim\tau$ and trivial otherwise. The choice of the isomorphism
depends on the orientation of $S^d$ which can be taken to be the
induced one from the orientations of $B$ and $\tau$.
\end{proof}

Consider the spectral sequence
\begin{equation}
E_1^{p,q}=H^{p+q}(\Gr^p_F C^\bullet(B,\Lambda)) \Rightarrow
H^{p+q}(B,\Lambda).
\label{spec-seq-ho-coho}
\end{equation}
It degenerates at $E_2$ because its $E_1$ page is concentrated in
$q=0$ by Lemma~\eqref{lemma-gradedcoho}. Let $d_1^{p,q}:E_1^{p,q}\ra
E_1^{p+1,q}$ denote the differential of the $E_1$ page.

\begin{proposition} 
\label{prop-map-ho-coho}
By Lemma~\eqref{lemma-gradedcoho}, we have  $$H^{n-i}(\Gr_F^{n-i}
C^\bullet(B,\Lambda)) =
\bigoplus_{{\tau\in\P,\tau\not\subset\partial B}\atop{\dim\tau=i}}
\Gamma(\tau,\Lambda)$$ and therefore an identification
\[
f:C_i(B,\partial B;\Lambda)\lra H^{n-i}(\Gr^{n-i}_F C^{\bullet}(B;\Lambda)).
\]
This turns into a map of complexes (varying $i$) when taking
$\partial_i$ and $d^{n-i,0}_1$ for the differentials respectively.
\end{proposition}

\begin{proof}  
We need to show that $f$ commutes with differentials, i.e. that
$d^{\bullet,\bullet}_{1}f=f\partial_\bullet$.  It will be sufficient
to show for an $i$-simplex $\tau$ and a facet $\omega$ of $\tau$
with $\omega\not\subset\partial B$ that for any element
$\alpha\in\Gamma(\tau,\Lambda)$, we have\\
%first column
\begin{minipage}[b]{0.7\textwidth}
\[
f((\partial_i\alpha)_\omega)=(d_1^{n-i,0}f \alpha)_\omega.
\]
where $(\partial_i\alpha)_\omega$ denotes the projection of
$\partial_i\alpha$ to ${\Gamma(\omega,\Lambda)}$ and similarly
$(d_1^{n-i,0}f \alpha)_\omega$ denotes the projection of
$d_1^{n-i,0}f \alpha$ to ${H^{n-i+1}C_\omega^{\bullet}(B,\Lambda)}$.
We first do the case $\Lambda=\ZZ$. Let $B_\sigma$ denote a suitably
embedded closed ball of dimension $n-\dim\sigma$ in $B$ meeting
$\sigma$ transversely (in a point) where $\sigma$ stands for $\tau$
or $\omega$. 
\end{minipage}
%second column
\begin{minipage}[b]{0.3\textwidth}
\qquad \includegraphics{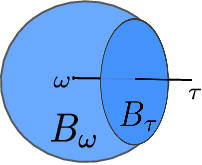}
\end{minipage}
We can arrange it such that $B_\tau$ is part of the boundary of
$B_\omega$, see the illustration above.  The point of this is that
the \v{C}ech complex $C^\bullet_\sigma(B,\ZZ)$ naturally computes
$H^{\bullet}_{B^\circ_\sigma}(B_\sigma,\ZZ)$ where $B_\sigma^\circ$
is the relative interior of $B_\sigma$. We claim that the component
\[
d_1^{n-i,0}: H^{n-i}C_\tau^{\bullet}(B,\ZZ) \lra
{H^{n-i+1}C_\omega^{\bullet}(B,\ZZ)} 
\]
is given by the sequence of maps
\[
H^{n-i}_{B^\circ_\tau}(B_\tau,\ZZ)\lra
H^{n-i}_{B^\circ_\tau}(\partial B_\omega,\ZZ) \lra H^{n-i}(\partial
B_\omega,\ZZ)\lra H^{n-(i-1)}_{B^\circ_\omega}(B_\omega,\ZZ).
\]
that are isomorphisms for $i<n$.
Indeed, a generator of $H^{n-i}(C_\tau^\bullet(B,\ZZ))$ is
represented by an element in 
\[
\bigoplus_{\left\{ I \left| {|I|=i+1, \tau\subset
U_I}\atop{\sigma\not\subset U_I\hbox{\tiny whenever
}\tau\subsetneq\sigma}\right.\right\}} \Gamma(U_I,\ZZ)
\] 
and this can be viewed as well as an element of
\[
\bigoplus_{\left\{ I \left| \omega\subset U_I,
|I|=i+1\right.\right\}} \Gamma(U_I,\ZZ_{\partial B_\omega})
\] 
where $\ZZ_{\partial B_\omega}$ denotes the constant sheaf supported
on $\partial B_\omega$.  The latter element gives an element
(actually a generator if $i<n$) of $H^{n-i}(\partial B_\omega,\ZZ)$
which then clearly maps to a generator of 
$H^{n-(i-1)}_{B^\circ_\omega} (B_\omega,\ZZ)$ under the \v{C}ech
differential
\[
\bigoplus_{\left\{ I \left| \omega\subset U_I,
|I|=i+1\right.\right\}} \Gamma(U_I,\ZZ_{\partial B_\omega})\lra 
\bigoplus_{\left\{ I \left| {|I|=i+2, \omega\subset
U_I}\atop{\sigma\not\subset U_I\hbox{\tiny whenever
}\omega\subsetneq\sigma}\right.\right\}} \Gamma(U_I,\ZZ).
\]
One checks that the orientations also match, so if $\omega$ has the
induced orientation from $\tau$ then the orientation of $B_\tau$ is
the induced one from $\partial B_\omega$ so there is no sign change
whereas there was one if this was opposite just as for the component
$\Gamma(\tau,\ZZ)\ra \Gamma(\omega,\ZZ)$ of $C_\bullet(B,\ZZ)$.  We
have thus proven the assertion for the case $\Lambda=\ZZ$.  The
general case follows directly as the component of the differentials
we considered is then just additionally tensored with the
restriction map $\Gamma(\tau,\Lambda)\ra \Gamma(\omega,\Lambda)$ in
the source as well as in the target of $f$.
\end{proof}

\begin{theorem}
\label{cor-iso-ho-coho}
The map $f$ induces a natural isomorphism \[
H_i(B,\partial B;\Lambda)\lra H^{n-i}(B;\Lambda).
\]
\end{theorem}

\begin{proof} 
By Proposition~\ref{prop-map-ho-coho} we thus obtain an isomorphism
$\Gr^\bullet_F H_i(B,\Lambda)\ra \Gr_F^\bullet H^{n-i}(B;\Lambda)$
where the filtration $F$ on homology is defined in the
straightforward manner.  We may remove $\Gr_F^\bullet$ from this map
because the graded pieces are concentrated in a single degree again
by Lemma~\ref{lemma-gradedcoho}.
\end{proof}

%===========================================================

		% end the bibliography

\end{document}